\documentclass[12pt]{amsart}
\usepackage{amscd,amssymb,amsthm,amsmath,amssymb,textcomp,multirow,rotating,longtable,mathrsfs,fancyhdr,textgreek,mathtools,wasysym,pdflscape,imakeidx,makecell, latexsym,amsthm,newlfont,enumerate}
\usepackage{tikz}

\usepackage{multicol}
\usepackage{tabulary}
\usepackage{colortbl}
\usepackage{tikz-cd}

\usepackage{geometry}
\DeclareMathOperator{\Pic}{Pic}
\DeclareMathOperator{\Cl}{Cl}
\DeclareMathOperator{\Exc}{Exc}
\DeclareMathOperator{\rk}{rk}
\DeclareMathOperator{\Bs}{Bs}
\DeclareMathOperator{\Sing}{Sing}

\DeclareMathOperator{\NEb}{\overline{NE}}
\DeclareMathOperator{\Nklt}{Nklt}
\DeclareMathOperator{\Effb}{\overline{Eff}}
\DeclareMathOperator{\ord}{ord}

\newcommand{\DR}{\mathbb{R}} 
\newcommand{\DC}{\mathbb{C}}
\newcommand{\DZ}{\mathbb{Z}}
\newcommand{\DQ}{\mathbb{Q}}
\newcommand{\DP}{\mathbb{P}}

\newcommand{\DN}{\mathbb{N}}

\newcommand{\Aut}{\mathrm{Aut}}

\newcommand{\PGL}{\mathrm{PGL}}

\newcommand{\modspace}[2]{\mathrm{M}^\mathrm{Kps}_{#1, #2}}
\newcommand{\morimukai}[2]{$\mathrm{MM}_{#1 - #2}$} 

\makeatletter
\newcommand*{\da@rightarrow}{\mathchar"0\hexnumber@\symAMSa 4B }
\newcommand*{\da@leftarrow}{\mathchar"0\hexnumber@\symAMSa 4C }
\newcommand*{\xdashrightarrow}[2][]{%
  \mathrel{%
    \mathpalette{\da@xarrow{#1}{#2}{}\da@rightarrow{\,}{}}{}%
  }%
}
\newcommand{\xdashleftarrow}[2][]{%
  \mathrel{%
    \mathpalette{\da@xarrow{#1}{#2}\da@leftarrow{}{}{\,}}{}%
  }%
}
\newcommand*{\da@xarrow}[7]{%
  \sbox0{$\ifx#7\scriptstyle\scriptscriptstyle\else\scriptstyle\fi#5#1#6\m@th$}%
  \sbox2{$\ifx#7\scriptstyle\scriptscriptstyle\else\scriptstyle\fi#5#2#6\m@th$}%
  \sbox4{$#7\dabar@\m@th$}%
  \dimen@=\wd0 %
  \ifdim\wd2 >\dimen@
    \dimen@=\wd2 %
  \fi
  \count@=2 %
  \def\da@bars{\dabar@\dabar@}%
  \@whiledim\count@\wd4<\dimen@\do{%
    \advance\count@\@ne
    \expandafter\def\expandafter\da@bars\expandafter{%
      \da@bars
      \dabar@ 
    }%
  }%
  \mathrel{#3}%
  \mathrel{%
    \mathop{\da@bars}\limits
    \ifx\\#1\\%
    \else
      _{\copy0}%
    \fi
    \ifx\\#2\\%
    \else
      ^{\copy2}%
    \fi
  }%
  \mathrel{#4}%
}
\makeatother

\newtheorem{theorem}{Theorem}[section]

\newtheorem{claim}{Claim}[section]

\newtheorem{lemma}[theorem]{Lemma}
\newtheorem{corollary}[theorem]{Corollary}
\newtheorem*{corollary*}{Corollary}
\newtheorem*{maincorollary*}{Main Corollary}

\newtheorem*{conjecture*}{Conjecture}

\newtheorem*{problem*}{Problem}
\newtheorem*{calabiproblem*}{Calabi Problem}
\newtheorem*{Main Theorem*}{Main Theorem}

\newtheorem*{theorem*}{Theorem}
\newtheorem*{maintheorem*}{Main Theorem}
\newtheorem*{maintheorem1*}{Main Theorem (I)}
\newtheorem*{maintheorem2*}{Main Theorem (II)}
\newtheorem*{maintheorem3*}{Main Theorem (III)}

\theoremstyle{definition}

\newtheorem*{example*}{Example}

\theoremstyle{remark}
\newtheorem{remark}[theorem]{Remark}
\newtheorem*{remark*}{Remark}

\makeatletter\@addtoreset{equation}{subsection} \makeatother

\newcommand{\Gm}{\mathbb{G}_m}

\usepackage{graphicx}
\usepackage{enumerate}
\usepackage{hyperref}
\usepackage{listings}
\usepackage{mathtools}
\usepackage{maple}
\usepackage[utf8]{inputenc}
\usepackage{textcomp}
\usepackage[matrix,arrow,curve]{xy}

\makeindex

\title{On $K$-stability of one-nodal prime Fano threefolds of genus $12$}

 \author{Elena Denisova}
 \address{\emph{Elena Denisova}
\newline
\textnormal{University of Edinburgh,  Edinburgh, Scotland}
\newline
\textnormal{\texttt{e.denisova@sms.ed.ac.uk}}}

 \author{Anne-Sophie Kaloghiros}
\address{\emph{Anne-Sophie Kaloghiros}
\newline
\textnormal{Brunel University London, Middlesex, England}
\newline
\textnormal{\texttt{Anne-Sophie.Kaloghiros@brunel.ac.uk}}}

\begin{document}

\maketitle
\dedicatory{To Professor Yuri Prokhorov on the occasion of his 60th birthday}

\begin{abstract}
We show that general one-nodal prime Fano threefolds of genus $12$ are K-polystable.     
\end{abstract}

\section{Introduction}
\subsection{Singular Fano threefolds of genus $12$}

Let $X$ be a Fano threefold with terminal Gorenstein singularities. By \cite{Nam97,JR11}, $X\hookrightarrow\mathcal X$ has a smoothing and $\mathcal{X}_t$ for $t\neq 0$ is a smooth Fano threefold with Picard rank $\rho(X)$ and anticanonical degree $(-K_X)^3$. Unless mentioned otherwise, a prime Fano threefold of genus $12$ will refer to a terminal Gorenstein Fano threefold with Picard rank $1$ and anticanonical degree $22$. Recent advances in the theory of K-stability show that there is a projective moduli space $\modspace{3}{22}$ whose closed points over $\DC$ parameterize K-polystable Fano threefolds of anticanonical degree $22$ that admit a smoothing (see \cite{Xu-book} as a reference on the general theory of K-moduli).

Let $X$ be a prime Fano threefold of genus $12$, then $X$ is $\DQ$-factorial precisely when $X$ is smooth \cite{Muk02}. Smooth prime Fano threefolds of genus $12$ form a $6$-dimensional family, which contains both K-polystable and strictly K-semistable members \cite[Section 7.1]{Fano21}. A precise description of which smooth prime Fano threefolds of genus $12$ are K-polystable or semistable is still conjectural. Denote by $\overline M$ the (non-empty $6$-dimensional) component of $\modspace{3}{22}$ parametrizing those $K$-polystable Fano threefolds of anticanonical degree $22$ with a smoothing to a prime Fano threefold of genus $12$. 

Prokhorov classifies prime Fano threefolds of genus $12$ with one node and shows that they form four $5$-dimensional families \cite{Pro16}. The goal of this note is to show: 
\begin{theorem}
    A general one-nodal prime Fano threefold of genus $12$ is K-polystable. 

    There are four boundary divisors of $\overline M$ parametrising K-polystable degenerations of one-nodal prime Fano threefolds of genus $12$. 
\end{theorem}

We now describe the geometry of the four families of one-nodal prime Fano threefolds of genus $12$ briefly. 
\begin{theorem}\label{fams}\cite{Pro16}
    Let $X$ be a prime Fano threefold of genus $12$, and assume that $\rk \Cl(X)= 2$ (or equivalently that $\Sing (X)$ consists of precisely one ordinary double point). Then $X$ is the midpoint of a Sarkisov link
\begin{center}
    \begin{tikzcd}
    & X_1 \arrow[rr, dashed, "\chi"] \arrow[dl,  "f_1"'] \arrow[dr,  "\pi_1"]  && X_2 \arrow[dl,  "\pi_2"'] \arrow[dr,  "f_2"]\\
    Z_1 && X && Z_2
    \end{tikzcd}
\end{center}
where $\pi_1$ and $\pi_2$ are small $\DQ$-factorializations, $\chi$ is a flop, and $f_1$ and $f_2$ are $K$-negative extremal contractions described as follows
\begin{enumerate}
    \item[(I)] $Z_1= \DP^3$ and $Z_2= \DP^3$, $f_1$ and $f_2$ are the blowups of curves $\Gamma_1\subset Z_1$ and $\Gamma_2\subset Z_2$ respectively. Both $\Gamma_1$ and $\Gamma_2$ are rational quintic curves that do not lie on quadric surfaces. 
    \item[(II)] $Z_1= Q\subset \DP^4$ and $Z_2= \DP^2$, $f_1$ is the blowup of a rational quintic curve $\Gamma_1 \subset Q$ that does not lie on a hyperplane section of $Q$, and $f_2$ is a conic bundle with discriminant $\delta$ of degree $3$.
    \item[(III)] $Z_1= V_5$ a quintic del Pezzo threefold, $Z_2= \DP^1$, $f_1$ is the blowup of a rational quartic curve $\Gamma \subset V_5$ and $f_2$ is a del Pezzo fibration of degree 6. 
    \item[(IV)] $Z_1= \DP^2$ and $Z_2= \DP^1$, $f_1 \colon \DP_{\DP^2}(\mathscr E)\to \DP^2$, where $\mathscr E$ is a stable rank 2 vector bundle and $f_2$ is a del Pezzo fibration of degree 5. 
\end{enumerate}

\end{theorem}

\begin{remark}
    These four families appear in \cite[Table 2]{KP23} under references (12na),(12nb),(12nc) and (12nd). 
    The blowup of a general member $X$ of Family (I) (resp.~(II), resp.~(III), resp.~(IV)) at its node is a weak Fano threefold $\widehat X$ whose anticanonical model admits a smoothing in Family \morimukai{2}{12} (resp.~\morimukai{2}{13}, resp.~\morimukai{2}{14}, resp.~\morimukai{3}{5}) in the classification of Fano threefolds \cite{MoMu81}.
\end{remark}
\begin{theorem}\cite[Proposition 5.66]{Fano21}
    There is a K-stable member of Family (IV). 
\end{theorem}
In this note, we prove:
\begin{theorem}
    There exist K-stable members of Families (I) and (II). There is a K-polystable member of Family (III).
\end{theorem}

\textbf{Acknowledgements.}
We thank Ivan Cheltsov, Yuchen Liu, Antoine Pinardin, Junyan Zhao and the anonymous referee for valuable discussions and comments. 
Anne-Sophie Kaloghiros was supported by EPSRC grant EP/V056689/1, Elena Denisova is a PhD student at the University of Edinburgh supported by the
School of Mathematics Studentship (funded via EPSRC DTP).

\section{Preliminary results on explicit K-stability of Fano threefolds}
All varieties considered are defined over $\DC$.
Let $X$ be a Fano variety with at most Kawamata log terminal
singularities of dimension $n\geq 2$, and let $G$ be a reductive subgroup in $\Aut(X)$. 
Let $\Xi$ be a divisor over $X$, that is $\Xi$ is a prime divisor on a normal variety $\widetilde X$ with a birational morphism $\varphi\colon \widetilde X\to X$. Define $\beta(\Xi)=A_X(\Xi)-S_X(\Xi)$, where
where $A_X(\Xi)= 1+ \ord_{\Xi}(K_{\widetilde{X}/ X})$ is the~log discrepancy of $\Xi$ and 
\[ S_X(\Xi) = \frac{1}{(-K_X)^n}\int_0^{\tau (\Xi)} \mathrm{vol}\big(\varphi^*(-K_X)-u\Xi\big)du \]
for $\tau(\Xi)= \sup\{ u\in \DR_{> 0} \vert \varphi^*(-K_X)-u\Xi \mbox{ is big }\}$.

\begin{theorem}\cite[Corollary~4.14]{Zhuang}\label{equivGstab}
\label{theorem:G-K -polystability}
Suppose that $\beta(\Xi)>0$ for every $G$-invariant  prime divisor $\Xi$ over $X$.
Then $X$ is $K$-polystable.
\end{theorem}
Recall the definition of the number $\alpha_{G,Z}(X)$, where $Z\subset X$ is a $G$-invariant subvariety:
$$
\alpha_{G,Z}(X)=\mathrm{sup}\left\{\lambda\in\mathbb{Q}\ \left|\ %
\aligned
&\text{the~pair}\ \left(X, \lambda D\right)\ \text{is log canonical at general point of $Z$ for any},\\
&\text{effective $G$-invariant $\mathbb{Q}$-divisor}\ D\ \text{on}\ X\ \text{such that}\ D\sim_{\mathbb{Q}} -K_X\\
\endaligned\right.\right\}.%
$$
Then $\alpha_G(X)\leqslant\alpha_{G,Z}(X)$.\index{$\alpha_{G,Z}(X)$}
\begin{lemma}\cite[1.44]{Fano21}\label{lemma144}
    Let $f\colon \widetilde X \to X$ be an arbitrary $G$-equivariant birational morphism, let $\Xi$ be a $G$-invariant prime divisor in $X$ such that $Z \subseteq f(\Xi)$, then we have
$$
\frac{A_X(\Xi)}{S_X(\Xi)}\geqslant\frac{n+1}{n}\alpha_{G,Z}(X).
$$
\end{lemma}
In particular, in dimension $3$, the existence of a $G$-invariant divisor $\Xi$ over $X$ with $\beta(\Xi)<0$ and $Z\subset c_X(\Xi)$ implies that $\alpha_{G,Z}(X)< \frac{3}{4}$, so that $Z$ is contained in $\Nklt(X, B_X)$ for some $B_X\sim_{\DQ} -\lambda K_X$ and rational number $\lambda < \frac{3}{4}$.

The next theorem is an application of the general inductive argument developed by Abban and Zhuang to bound the ratio $\frac{A_X(\Xi)}{S_X(\Xi)}$ \cite{AZ22} to the case of smooth Fano threefolds. 
\begin{theorem}\cite[Corollary 1.110]{Fano21}
\label{corollary:Kento-formula-Fano-threefold-surface-curve}
Let $X$ be a smooth Fano threefold, let $Y$ be an irreducible normal surface in the~threefold $X$,
let $Z$ be an irreducible curve in $Y$, and $\Xi$ a~prime divisor over $X$ with $C_X(\Xi)=Z$. Then
\begin{equation*}
\label{equation:Kento-formula-Fano-threefold-surface-curve}
\frac{A_X(\Xi)}{S_X(\Xi)}\geqslant\min\Bigg\{\frac{1}{S_X(Y)},\frac{1}{S\big(W^Y_{\bullet,\bullet};Z\big)}\Bigg\}
\end{equation*}
and
\begin{multline*}
    S\big(W^Y_{\bullet,\bullet};Z\big)=\frac{3}{(-K_X)^3}\int_0^\tau\big(P(u)^{2}\cdot Y\big)\cdot\ord_{Z}\Big(N(u)\big\vert_{Y}\Big)du+\\+\frac{3}{(-K_X)^3}\int_0^\tau\int_0^\infty \mathrm{vol}\big(P(u)\big\vert_{Y}-vZ\big)dvdu 
\end{multline*}

where $P(u)$ is the~positive part of the~Zariski decomposition of the~divisor $-K_X-uY$, and $N(u)$ is its negative part.
\end{theorem}

\begin{remark}
Here, $W^Y_{\bullet,\bullet}$ is a $\DN^2$ linear series defined as the refinement of the anticanonical ring \[V^X_{\bullet}= \bigoplus_{m\in \DN} H^0(X, -mK_X)\] by the divisor $Y$. We refer to \cite[Section 2]{AZ22} or \cite[Section 1.7]{Fano21} for the definition of $W^Y_{\bullet,\bullet}$ and of the associated invariant $S\big(W^Y_{\bullet,\bullet};Z\big)$. We take the expression in Theorem~\ref{corollary:Kento-formula-Fano-threefold-surface-curve} as a definition of $S\big(W^Y_{\bullet,\bullet};Z\big)$. Note that an expression for $S\big(W^Y_{\bullet,\bullet};Z\big)$ can be computed in the more general context of $\DQ$-factorial Mori Dream spaces \cite[Theorem 1.106]{Fano21}.
\end{remark}
We recall a few results on nonklt centres of pairs $(X,B_X)$ where $X\sim -\lambda K_X $ for $\lambda \in \DQ$ when $X$ admits morphisms to projective spaces. 
\begin{lemma}[{\cite[Corollary A.10]{Fano21}}]\label{CorollaryA13}
Suppose $X=\mathbb{P}^3$ and $B_X\sim_{\mathbb{Q}} -\lambda K_X$
for some rational number $\lambda<\frac{3}{4}$.
Let $Z$ be the~union of one-dimensional components of  $\Nklt(X,B_X)$.
Then $\mathcal{O}_{\mathbb{P}^3}(1)\cdot Z\leqslant 1$. In particular, if $Z\ne 0$, then $Z$ is a line.
\end{lemma}
\begin{lemma}[{\cite[Corollary A.12]{Fano21}}]\label{CorollaryA15}
Suppose that $X$ is a~smooth Fano threefold, $B_X\sim_{\mathbb{Q}} -\lambda K_X$ for some rational number~\mbox{$\lambda<1$},
and there exists a~surjective morphism with connected fibers \mbox{$\phi\colon X\to \mathbb{P}^1$}.
Set $H=\phi^*(\mathcal{O}_{\mathbb{P}^1}(1))$. Let $Z$ be the~union of one-dimensional components of $\mathrm{Nklt}(X,\lambda B_X)$.
Then $H\cdot Z\leqslant 1$.
\end{lemma}

\begin{lemma}[{\cite[Corollary A.13]{Fano21}}]\label{CorollaryA16}
\label{corollary:Nadel-bound-for-CB}
Suppose that $-K_{X}$ is nef and big, $B_X\sim_{\mathbb{Q}} -\lambda K_X$
for some rational number~\mbox{$\lambda<1$}, and there exists a~surjective morphism with connected fibers \mbox{$\phi\colon X\to \mathbb{P}^2$}.
Set $H=\phi^*(\mathcal{O}_{\mathbb{P}^2}(1))$. Let $Z$ be the~union of one-dimensional components of  $\mathrm{Nklt}(X,\lambda B_X)$.
Then $H\cdot Z\leqslant 2$.
\end{lemma}

\section{Family I}\label{FamilyI}
Let $X$ be a one-nodal prime Fano threefold of genus $12$ that belongs to Family I of Theorem~\ref{fams}, then $X$ is the midpoint of a Sarkisov link associated to a Cremona transformation $\DP^3\dashrightarrow\DP^3$ which is a degeneration of the cubo-cubic transformation \cite{BL12}.  
We describe the associated birational geometry briefly, see \cite{BL12, CM13, Pro16} and \cite{KP23} for proofs and precise statements. 

\begin{figure}[h!]
    \centering
    \begin{tikzcd}&& \widehat X \arrow[dl,  "\sigma_1"'] \arrow[dr,  "\sigma_2"] \arrow[dd,  "\pi"]  &&\\
    & X_1  \arrow[dl,  "f_1"'] \arrow[dr,  "\pi_1"]  && X_2 \arrow[dl,  "\pi_2"'] \arrow[dr,  "f_2"]\\
    \DP^3 && X && \DP^3
    \end{tikzcd}
\end{figure}

Denote by $H_i= \sigma_i^*\big(f_i^*\mathcal{O}_{\DP^3}(1)\big)$ for $i=1,2$, and by $H= \pi^*(-K_X)$ the pullbacks to $\widehat X$ (or to any of the models) of the ample generators of $\Pic(\DP^3)$ and of $\Pic (X)$. Given a curve $\Gamma\subset \DP^3$, we (sloppily) denote by $|nH_1-\Gamma|$ the linear system $H^0(\DP^3, \mathcal O_{\DP^3}(n)\otimes \mathcal I_{\Gamma})$ of surfaces of degree $n$ on which $\Gamma$ lies. 
The morphism $f_1$ is the blowup of a smooth rational quintic curve $\Gamma_1\subset \DP^3$ that does not lie on a quadric ($|2H_1-\Gamma_1|= \emptyset$), and there is a unique quadrisecant line $L_1$ to $\Gamma_1$. 
The curve $\Gamma_1$ lies on a cubic surface, $|3H_1-\Gamma_1|$ has dimension $4$ and $\Bs |3H_1-\Gamma_1|= \Gamma_1\cup L_1$. The birational map associated to $|3H_1-\Gamma_1| = |H_2|$ is precisely the Cremona transformation $\DP^3\dashrightarrow \DP^3$ induced by the Sarkisov link above. The threefold $X_1$ is weak Fano, 
\begin{equation*}
     -K_{X_1}\sim H\sim 4H_1-E_1,
\end{equation*}
where $E_1= \Exc f_1$, so that the proper transform of $L_1$ (still denoted $L_1$) is the unique flopping curve on $X_1$. The map $\pi_1$ contracts $L_1$ to a node $\{x_0\}= \operatorname{Sing}(X)\in X$. 

Let $\pi\colon \widehat X\to X$ be the blowup of $x_0$, and $\sigma_1$ the induced map to $X_1$. Note that $X_1$ and $X_2$ are the two small resolutions of the node $x_0$, $\chi\colon X_1\dashrightarrow X_2$ is the associated Atiyah flop and $L_1= \sigma_1(F)$, where $F= \Exc \pi$. 
Then, $\widehat X$ is a weak Fano threefold of $\rho=3$ and we have \cite{KP23}:
\begin{equation*}\label{eq2}
 -K_{\widehat X}\sim H-F \sim H_1+H_2   
\end{equation*}
and from 
\[ H\sim 4H_1-E_1 \sim 4H_2-E_2\]
we deduce
\begin{equation*}\label{eq3}
    H\sim 2(H_1+H_2)- \frac{E_1+E_2}{2} \quad \mbox{and} \quad H_1+H_2 \sim \frac{E_1+E_2}{2}+F.
\end{equation*}

For future reference, let $T_1$ be a cubic surface containing $\Gamma_1$, and denote by $T$ its proper transform on $\widehat{X}$. Since $\Bs|3H_1-\Gamma_1|= \Gamma_1\cup L_1$, $L_1$ lies on $T_1$ and:
\[ T\sim 3\sigma_1^*(f_1^* \mathcal O_{\DP^3}(1))-E_1-F\sim 3H_1-E_1-F\sim H_2,\]
so that
\[ -K_{\widehat X}-uT\sim H_1+H_2-uH_2 \in \mathbb Z_{\geq 0}[H_1]+ \mathbb Z_{\geq 0}[H_2]\subset \operatorname{Nef}(\widehat X) \]
is nef for $0\leq u\leq 1$. 
For $u>1$, $-K_{\widehat X}-uT$ is no longer nef. If $C$ is the proper transform on $\widehat X$ of a minimal rational curve contracted by $f_1$, then $H_1\cdot C=0$ and $H_2\cdot C>0$, so that 
\[ -K_{\widehat X}-uT\sim H_1\cdot C-(u-1)H_2\cdot C<0.\]

We may write for $u\geq 1$
\begin{eqnarray*} -K_{\widehat X}-uT \sim uH_1-(u-1)(H_1+H_2)
&\sim uH_1 -(u-1)(4H_1-E_1-F)\\
&\sim (4-3u) H_1 +(u-1)(E_1+F)\end{eqnarray*}
showing that the pseudo-effective threshold is $u= \frac{4}{3}$, and that $-K_{\widehat X}-uT $ admits a Zariski decomposition with nef positive part $P(u)= (4-3u) H_1 $ and negative part $N(u)= (u-1)(E_1+F)$. 

\subsection{Construction of a member with $\DZ_2\times \DZ_2$-action} We now consider a special member of Family (I).  
Let $C_{(a,b)}$ be the image of the embedding $\DP^1\hookrightarrow\DP^3$ given by
$$[x:y]\to [x^5: ax^4y+bx^2y^3:bx^3y^2+axy^4:y^5]\text{  for }a,b\in \DC^* ;$$ then $C_{(a,b)}$ is a rational quintic curve that does not lie on a quadric surface for $|a|\ne |b|$. The curve $C_{(a,b)}$ is invariant under the action of $G:=\DZ/2\DZ\times\DZ/2\DZ$ on $\DP^3$ defined by:
\begin{align*}
    &\tau_1:[x_0:x_1:x_2:x_3]\to [x_3:x_2:x_1:x_0],\\
    &\tau_2:[x_0:x_1:x_2:x_3]\to [x_0:-x_1:x_2:-x_3].
\end{align*}
In fact, the action of $\tau_1$ (resp.~$\tau_2$) on $C_{(a,b)}$ is induced by that of the involution of $\DP^1$ given by $[x:y]\leftrightarrow [y:x]$ (resp.~$[x:y]\leftrightarrow [x:-y]$).
We consider the element of Family (I) obtained by taking the curve $\Gamma_1= C_{(1,-4)}$.

Since $\Gamma_1$ is $G$-invariant, $L_1$ is also $G$-invariant and $X_1$ and $X$ are endowed with a $G$-action. 
\begin{claim}
 The group $\Aut(X)$ is finite.    
\end{claim}
\begin{proof}
The curve $\Gamma_1$ is not contained in a hypersurface of $\DP^3$, the stabilizer of $\Gamma$ in $\Aut(\DP^3)$ is $\Aut(\DP^3; \Gamma_1) \simeq \Aut(\Gamma_1)\simeq \Aut(\DP^1)$. By construction of $X$, $\Aut(X)$ is a subgroup of the group $\Aut(\DP^3, \Gamma_1)\simeq \Aut(\DP^1)$ that preserves the four points of intersection $\Gamma_1\cap L_1$, so it is a finite group.  
\end{proof}
We will apply Theorem~\ref{equivGstab} to prove that $X$ is K-stable. To do so, we first describe possible centres of $G$-invariant divisors over $X$. In what follows, $\Xi$ always denotes a $G$-invariant prime divisor over $X$.

\begin{claim}\label{cla1}
If the centre of $\Xi$ on $X$ is $0$-dimensional, it is the singular point $c_X(\Xi)= \{x_0\}$.     
\end{claim}
    
\begin{proof}
    There is no point of $\DP^3$ fixed by the action of $G$.
\end{proof}

\begin{claim}\label{cla2}
If the centre of $\Xi$ on $\DP^3$ is a line $L$, then  
\[ L=L_{[\lambda: \mu]}=\begin{cases} \lambda x_0+\mu x_2=0,\\ \lambda x_3+\mu x_1=0. \end{cases}.\]
All $G$-invariant lines lie on the quadric $Q= \{ x_1x_0-x_2x_3=0\}$. Any two distinct $G$-invariant lines are disjoint. A $G$-invariant line $L\neq L_1$ is either disjoint from $\Gamma_1$ or meets $\Gamma_1$ in precisely two points.  
\end{claim}

\begin{proof}
Let $L\subset \DP^3$ be a G-invariant line, and consider any two distinct hyperplanes $H_1=\{f_1=0\}$ and $H_2= \{ f_2=0\}$ containing $L$, so that $L= H_1\cap H_2= \{ f_1=f_2=0\}$. Then, $L= \Bs \mathscr H$ is the base locus of the pencil $\mathscr H= \{ u f_1+v f_2=0; [u:v] \in \DP^1\}$. 

The line $L= \Bs \mathscr H$ is $G$-invariant precisely when $G$ fixes $\mathscr H$, or equivalently when both $\tau_1$ and $\tau_2$ induce involutions on $\mathscr H$, and on its base $\DP^1$. Up to reparametrizing the pencil $\mathscr H$ we may assume that $[u:v]= [1:0]$ is a $\tau_2$-invariant hyperplane, that is, the linear form $f_1(x_0, \cdots , x_3)$ is one of
\[ \lambda x_0+ \mu x_2 \mbox{ or } \lambda x_3+ \mu x_1 \mbox{ for } [\lambda:\mu]\in \DP^1,\]
and $H_1= \{ \lambda x_0+\mu x_2=0\}$ or $H_1= \{ \lambda x_3+\mu x_1=0\}$.
The condition that $\mathscr H$ is $G$-invariant is then that $\tau_1\cdot H_1$ is a fibre of the pencil, so that (noting that $H_1$ is not fixed by $\tau_1$)
\[ L= L_{[\lambda: \mu]}= H_1\cap \tau_1\cdot H_1= \begin{cases}
    \lambda x_0+ \mu x_2=0,\\
    \lambda x_3+ \mu x_1=0.
\end{cases}\]
which gives the desired expression. 

Check that $L_{[\lambda: \mu]}\subset Q$ for all $[\lambda: \mu] \in \DP^1$, that $L_{[\lambda: \mu]}\cap L_{[\lambda':\mu']}= \emptyset$ for $[\lambda: \mu]\neq [\lambda': \mu']$, and that $L_{[\lambda:\mu]}\cap \Gamma_1= \emptyset$ unless $[\lambda:\mu]\in \{ [0:1], [3:1], [-5:1]\}$ and  $L_{[\lambda:\mu]}\cap \Gamma_1$ consists of 2 points, or $[\lambda:\mu]= [1:1]$ and $L_{[1:1]}= L_1$ is the unique quadrisecant to $\Gamma_1$.
\end{proof}

\begin{remark}
    Given that the Sarkisov link of which $X$ is a midpoint is $G$-equivariant, $E_2$ and $\Gamma_2= \Gamma^+$ are also invariant under the induced $G$-action. Since the map $\DP^3\dashrightarrow\DP^3$ is induced by $|H_2| = |3H_1-\Gamma_1|$, the fibres of $E_2\to \Gamma_2$ are the transforms of trisecant lines of $\Gamma_1$. Since none of these are $G$-invariant, the action of $G$ on $\Gamma_2$ does not fix $\Gamma_2$ pointwise either. 
\end{remark}

\begin{claim}\label{cla3}
Let $H_{[\lambda:\mu]}$ be a general hyperplane containing $L_{[\lambda: \mu]}$. 
    
    Then $H_{[\lambda: \mu]}\cap \Gamma_1= \{ b_1, \cdots , b_5\}$ and $H_{[\lambda: \mu]}\cap L_1= \{ b_0\}$, where $b_1, \cdots , b_5$ (resp. $b_0, \cdots , b_5$) consists of $5$ (resp.~$6$) points in general position.  
\end{claim}
    \begin{proof}
       Fix $[\lambda: \mu]\in \DP^1$, and let $\mathscr H$ be the pencil of hyperplanes containing $L_{[\lambda:\mu]}$. We compute that the general fibre of $\mathscr H$ intersects $\Gamma_1\cup L_1$ in $6$ distinct points. Assume that for some fibre $H$ of $\mathscr H$, $3$ of the $5$ points of $H\cap \Gamma_1$ lie on a line (resp.~the $6$ points $H\cap \big(\Gamma_1\cup L_1\big)$ lie on a conic). Then, this line (resp.~conic) is contracted by the Cremona transformation $\DP^3\dashrightarrow \DP^3$ to a point lying on $\Gamma_2$. If the points of intersection of a general hyperplane containing $L_{[\lambda: \mu]}$ are not in general position, then we define a dominant rational map $\DP^1\dashrightarrow \Gamma_2$ from the base of $\mathscr H$ to $\Gamma_2$, leading to a contradiction.

    \end{proof}

We now turn to the proof that no $G$-invariant prime divisor $\Xi$ over $X$ with $\beta(\Xi)\leq 0$ has $1$-dimensional centre $Z=c_{\DP^3}(\Xi)$. If $Z$ is $1$-dimensional, then either $Z= \Gamma_1$, or by Lemma~\ref{lemma144}, $Z$ is the union of $1$-dimensional components of $\Nklt(\DP^3, B)$ for some $B\sim \mathcal{O}_{\DP^3}(4\lambda)$ with $\lambda\in \DQ$, $\lambda<3/4$. Then, by Lemma~\ref{CorollaryA13}, $Z$ can only be a line. 

\begin{lemma}\label{dim1a}
    If $Z=c_{\DP^3}(\Xi)$ is a $G$-invariant line distinct from $L_1$, $\beta(\Xi)>0$. 
\end{lemma}
\begin{proof}
We will use Lemma~\ref{CorollaryA13} to find a lower bound for $\beta(\Xi)$. To this effect, we find an irreducible normal surface $S\subset \widehat X$ containing $Z$ and use the inequality
\begin{equation*}
\frac{A_{X}(\Xi)}{S_{X}(\Xi)}\geqslant\min\Bigg\{\frac{1}{S_{X}(S)},\frac{1}{S\big(W^S_{\bullet,\bullet};Z\big)}\Bigg\}.
\end{equation*}

Let $S_1\subset X_1$ and $S\subset \widehat X$ be the pullbacks to $X_1$ and $\widehat X$ of a general hyperplane containing $Z$. By Claim~\ref{cla3}, $S_1\subset X_1$ is a del Pezzo surface of degree $4$ and $S$ is a cubic surface. Recall that $E_1$ and $E_2$ are the $f_1$ and $f_2$ exceptional divisors, and that $F$ is the $\pi$-exceptional divisor. All these are $G$-invariant, and $E_2$ is covered by the (proper transforms of) trisecant lines of $\Gamma_1$. 

We first compute $S_X(S)$; on $\widehat X$, we have the following intersection numbers: 
\begin{align*}
    &S^3=1,& &S^2\cdot E_1 = 0, & &S\cdot E_1^2=-5, & &E_1^3=-18,\\
    &      & &S^2\cdot F = 0, & &S\cdot F^2=-1, & &F^3=2,\\
    &      & &S\cdot E_1\cdot F=0,& &E_1\cdot F^2=-4, & &E_1^2\cdot F=0.
\end{align*}
and relations \cite{KP23}
\begin{align*}
    S\sim H_1\sim 3H_2-E_2-F & &  H\sim 4H_1-E_1\sim 4H_2-E_2 && H-F\sim H_1+H_2. 
\end{align*}
Define, for $u\geq 0$, the divisor
\begin{eqnarray*}
    D_u= \pi^*(-K_X)-uS\sim (1-u) H+ u(H-S) \sim (1-u)H+ u(H_2+F)=\\
    = (1-u)H+ u\big(\frac{H+E_2}{4}+F\big),
\end{eqnarray*}
then $D_u$ is pseudo-effective for $u \leq 4/3$, and has a Zariski decomposition with nef positive part
\[ D_u= P(u)+N(u) \mbox{, where }  \begin{cases} P(u)=\big(1-\frac{3u}{4}\big) H,\\ N(u)= u(\frac{E_2}{4}+F).\end{cases}\]and we compute:
\begin{align*}
   S_{X}(S)&=\frac{1}{(-K_X)^3}\int_{0}^{\tau}\mathrm{vol}(\pi^*(-K_X)-uS)du= \frac{1}{22}\int_{0}^{4/3}\frac{11(4 - 3u)^3}{32}du=\frac{1}{3}.
\end{align*}

We now show that $S\big(W^S_{\bullet,\bullet};Z\big)<1$ in order to apply Theorem~\ref{corollary:Kento-formula-Fano-threefold-surface-curve}. 

The surface $S$ is a cubic surface obtained by blowing up a general hyperplane $\DP^2$ at 6 points $\{ b_0, \cdots , b_5\}$ in general position. Let $\ell$ be the pullback of the generator of $\Pic (\DP^2)$, and $e_0, \cdots, e_5$ the exceptional curves. The Mori cone $\NEb(S)$ is generated by $e_0, \cdots , e_5$, by the proper transforms $l_{i,j}= \ell-e_i-e_j$ of lines through two of the blownup points, and by the proper transforms $q_i= 2l-\sum e_j+e_i$ of the conics through any 5 of the blownup points $\{b_0, \cdots, b_5\}$. 

We want to evaluate \begin{multline*}
    S\big(W^S_{\bullet,\bullet};Z\big)=\frac{3}{(-K_X)^3}\int_0^\tau\big(P(u)^{2}\cdot S\big)\cdot ord_{Z}\Big(N(u)\big\vert_{S}\Big)du+\\+\frac{3}{(-K_X)^3}\int_0^\tau\int_0^\infty \mathrm{vol}\big(P(u)\big\vert_{S}-vZ\big)dvdu. 
\end{multline*}

Since there are no $G$-fixed points, $Z\subset S$ is one of $\ell$ or the lines $l_{i,j}$. Recall that $E_2\sim 8H_1-3E_1-4F$, and restricting to $S$ gives $E_2\big\vert_S\sim 8\ell- 3(e_1+ \cdots + e_5)-4e_0$. From the description of $\NEb(S)$,   
\[
\ord_{Z}\big (E_2\big\vert_{S}\big)\le 2\] and
$$\ord_{Z}\Big(N(u)\big\vert_{S}\Big)\le
\begin{cases}
    2\cdot \frac{1}{4}= \frac{1}{2}\text{ if }Z\subset E_2,\\
    0 \text{ otherwise.} 
\end{cases}$$

The first term of the expression $S\big(W^S_{\bullet,\bullet};Z\big)$ is bounded by
$$\frac{3}{(-K_X)^3}\int_0^\tau\big(P(u)^{2}\cdot S\big)\cdot\ord_{Z}\Big(N(u)\big\vert_{S}\Big)du\le \frac{3}{22}\int_0^{4/3}  \frac{11(4 - 3u )^2}{16}\cdot \frac{1}{2} du = \frac{1}{3}.$$

\noindent {\bf Case 1:  $Z\cap \Gamma_1= \emptyset$.} In this case, $Z\sim \ell$, and the Zariski decomposition of $$(-\pi^*(K_X)-uS)|_S-vZ= P(u,v)+N(u,v);$$ for $u\in\big[0,\frac{4}{3}\big]$ is given by 
\begin{align*}
     && N(u,v)=
    \begin{cases}
        0& \mbox{ for } 0 \le v\le \frac{3(4-3u)}{8},\\
        \frac{8v-3(4-3u)}{4}q_0 & \mbox{ for }\frac{3(4-3u)}{8}\le v \le \frac{4-3u}{2}.
    \end{cases}
\end{align*}
and we compute:
\begin{align*}
    &&P(u,v)^2=
    \begin{cases}
         v^2 - 8v + 6uv + \frac{99u^2}{16}  - \frac{33u}{2} +11& \mbox{ for } 0 \le v\le \frac{3(4-3u)}{8},\\
        \frac{5(2v - (4-3u))^2}{4}& \mbox{ for }\frac{3(4-3u)}{8}\le v \le \frac{4-3u}{2}.
    \end{cases}
\end{align*}
This yields:
{\allowdisplaybreaks\begin{align*}
    S\big(W^S_{\bullet,\bullet};Z\big)&\le \frac{1}{3}+\frac{3}{(-K_X)^3}\int_0^\tau\int_0^\infty \mathrm{vol}\big(P(u)\big\vert_{S}-vZ\big)dvdu\le \\
    &\le \frac{1}{3}+\frac{3}{22}\int_0^{4/3}\Bigg( \int_0^{\frac{3(4-3u)}{8}} v^2 - 8v + 6uv + \frac{99u^2}{16}  - \frac{33u}{2} +11 dv +\\ &\text{ }\text{ }\text{ }\text{ }+\int_{\frac{3(4-3u)}{8}}^{\frac{4-3u}{2}} \frac{5(2v - (4-3u))^2}{4} dv\Bigg) du\le  \frac{1}{3}+\frac{53}{132} =\frac{97}{132}<1,
\end{align*}}
which is what we wanted. 

\noindent {\bf Case 2:  $Z\cap \Gamma_1\neq \emptyset$.} As $Z$ is one of the bisecant lines of $\Gamma_1$,   $Z\sim l_{i,j}= \ell-e_i-e_j$ for some $1\leq i<j\leq 5$. We may assume that $Z\sim l_{1,2}$. 
Write the Zariski decomposition of $(-\pi^*(K_X)-uS)|_S-vZ$ for $0 \le u\le \frac{4}{3}$ we have:
\begin{align*}
     &&N(u,v)=
    \begin{cases}
        0&\mbox{ for }0 \le v\le \frac{3(4-3u)}{4},\\
        \frac{4v-3(4-3u)}{4}(e_1+e_2)&\mbox{ for }\frac{3(4-3u)}{4}\le v\le \frac{4-3u}{2},\\
        \frac{4v-3(4-3u)}{4}(e_1+e_2)+\frac{2v-3(4-3u)}{2}(\ell_{34}+\ell_{35}+\ell_{45})&\mbox{for } \frac{4-3u}{2}\le v\le  \frac{5(4-3u)}{8}.
    \end{cases}
\end{align*}
This time, we compute:
{\allowdisplaybreaks\begin{align*}
    S\big(W^S_{\bullet,\bullet};Z\big)&\le \frac{1}{3}+\frac{3}{(-K_X)^3}\int_0^\tau\int_0^\infty \mathrm{vol}\big(P(u)\big\vert_{S}-vZ\big)dvdu\le \\
    &\le \frac{1}{3}+\frac{3}{22}\int_0^{4/3}\Bigg( \int_0^{\frac{4-3u}{4}} - v^2 - 4v + 3uv+  \frac{99u^2}{16}- \frac{33u}{2} +11 dv +\\ 
    &\text{ }\text{ }\text{ }\text{ }+\int_{\frac{4-3u}{4}}^{\frac{4-3u}{2}} v^2 - 8v + 6uv + \frac{117u^2}{16}- \frac{39u}{2} +13 dv +\\ 
    &\text{ }\text{ }\text{ }\text{ }+\int_{\frac{4-3u}{2}}^{\frac{5(4-3u)}{8}} \frac{( 8v - 5(4-3u))^2}{16} dv \Bigg) du\le  \frac{1}{3}+\frac{23}{44} =\frac{113}{132}<1.
\end{align*}}
and this finishes the proof in this case.
\end{proof}

\begin{lemma}\label{dim1b}
    Let $\Xi$ be a prime divisor over $X$ with $c_{\DP^3}(\Xi)= L_1$ then $\beta(\Xi)>0$. 
\end{lemma}
\begin{proof}

  By \cite{KP23}, there are precisely 4 lines through $x_0\in X$. Since $G$ fixes $x_0$ and $G$ sends lines to lines, $G$ sends a line through $x_0$ to a line through $x_0$.  If $L\ni x_0$ is a line and $\widehat{L}$ is its proper transform on $\widehat X$, $-K_{\widehat X}\cdot \widehat L=0$ and $\widehat L$ is a flopping curve.  Let $\omega\colon \widetilde{X}\to \widehat X$ be the blowup of the proper transforms of the 4 lines through $x_0\in X$, and denote by $\Lambda= \Lambda_1+ \Lambda_2+\Lambda_3+ \Lambda_4$ its ($G$-invariant) exceptional divisor. Denote by $\widetilde F= \omega^* F$, $\widetilde E_1= \omega^* E_1-\Lambda$ the proper transforms of $F$ and $E_1$ on $\widetilde{X}$. On $\widetilde X$, we have the intersection numbers: \begin{align*}
  &\Lambda^3=8, & & \Lambda^2\cdot \omega^* F = -4, & & \Lambda\cdot \omega^*(F)^2 = 0, & & \Lambda \cdot \omega^*(F)\cdot \omega^*(E_1) =0 ,\\
    &  \Lambda^2 \cdot  \omega^*(E_1) = 4, & & \Lambda \cdot  \omega^*(E_1)^2 =0, & & \omega^*F^3= 2,&& \omega^*E_1^3= -18.\\
    &      & &
\end{align*}

We first show that 
the Zariski decomposition of $\omega^*\pi^* (-K_X)- u\widetilde F$ exists and writes $P(u)+N(u)$, where $P(u)$ is nef and 
\begin{align*}
   &&N(u)=
    \begin{cases}
   0&\mbox{ for }0\le u \le 1,\\
    (u-1)\Lambda& \mbox{ for }1\le u \le 3.
    \end{cases}
\end{align*}
We have $A_X(\widetilde F)=2$ and compute
\begin{align} \label{eq:EL}
   S_X(\widetilde F)&=\frac{1}{(-K_X)^3}\int_{0}^{\tau}\mathrm{vol}(\omega^*\pi^*(-K_X)-u\widetilde F) du=\\
   &=\frac{1}{22}\Bigg(\int_{0}^{1} 22-2u^3 du + \int_{1}^{3} 2(u - 3)(u^2 - 3u - 3) du \Bigg)=\frac{83}{44}.
\end{align}
So that $\beta(\widetilde F)> 0$. 

 Now we assume that the centre of $\Xi$ over $\widetilde X$ is one-dimensional, so that $Z= c_{\widehat X}(\Xi)\subset \widetilde F$ is an irreducible curve. The surface $\widetilde F$ is the blowup of $F\simeq \DP^1\times \DP^1$ at 4 points, so it is a del Pezzo surface of degree 4. Let $\ell_1$ and $\ell_2$ be the pullbacks to $\widetilde F$ of the two rulings and $e_1, e_2,e_3,e_4$ the exceptional divisors. Then $\NEb(\widetilde F)$ is generated by the proper transforms of rulings through one of the blownup points ($\ell_{1,i}=\ell_1 -e_i$ or $\ell_{2,i}=\ell_2-e_i$) 
and by the proper transforms of $(1,1)$ curves on $F$ through 3 blownup points $\ell_{i,j,k}=\ell_1+\ell_2-e_i-e_j-e_k$. 
We use Theorem~\ref{corollary:Kento-formula-Fano-threefold-surface-curve} to find a lower bound for $\beta(\Xi)$. 
We have
\begin{multline*}
S\big(W^{\widetilde F}_{\bullet,\bullet};Z\big)=\frac{3}{(-K_X)^3}\int_0^{3}\big(P(u)^{2}\cdot \widetilde F\big)\cdot\ord_{Z}\Big(N(u)\big\vert_{\widetilde F}\Big)du+\\+\frac{3}{(-K_X)^3}\int_0^{3}\int_0^\infty \mathrm{vol}\big(P(u)\big\vert_{\widetilde F}-vZ\big)dvdu.
\end{multline*}
Since $Z\not \in \{e_1,e_2,e_3,e_4\}= \Lambda_{\widetilde F}$, $\ord_{Z}\big(N(u)\big\vert_{\widetilde F}\big)=0$. By construction, we may write
$$Z\sim \alpha_1e_1 +\alpha_2e_2+\alpha_3e_3+\alpha_3e_4+ \sum_{\substack{i\in \{1,2,3,4\},\\ j\in \{1,2\}}} \alpha_{ij}\ell_{i(j)}+ \alpha_{123}\ell_{123}+ \alpha_{124}\ell_{124}+ \alpha_{134}\ell_{134}+ \alpha_{234}\ell_{234}.$$
Since there is no $G$-fixed point in $\DP^3$ on either side of the link, one of $\alpha_{123}$, $\alpha_{124}$, $\alpha_{134}$, or $\alpha_{234}$ is greater than $1$. Without loss of generality we assume that $\alpha_{123}\ge 1$; by convexity of volume we get the inequality 
$$
S\big(W_{\bullet,\bullet}^{\widetilde F};Z\big)=\frac{3}{22}\int_{0}^{3}\int_{0}^{\infty}\mathrm{vol}\Big(P(u)\big\vert_{\widetilde F}-vZ\Big)dvdu\leqslant \frac{3}{22}\int_{0}^{3}\int_{0}^{\infty}\mathrm{vol}\Big(P(u)\big\vert_{\widetilde F}-v\ell_{123}\Big)dvdu.
$$
so it is enough to show that the last integral is less than $1$ to conclude.
\par We now assume $Z\sim \ell_{123}$, and denote by $P(u,v)$ and $N(u,v)$ the positive and negative parts of the Zariski decomposition of  $(\omega^*\pi^*(-K_X)-u\widetilde F)|_{\widetilde F}-vZ$. Then
\begin{itemize}
    \item if $u\in [0,1]$ then for $0\le v\le u$
       $N(u,v)=v(e_1 + e_2 + e_3)$ so that $
        P(u,v)^2=2(u - v)^2$.
    \item if $u\in [1,2]$ then:
    {\allowdisplaybreaks\begin{align*}
        &N(u,v)=
        \begin{cases}
        0&\mbox{ for } 0\le v\le u-1,\\
        (-u + v + 1)(e_1 + e_2 + e_3)&\mbox{ for } u-1\le v\le 1 ,\\
        (-u + v + 1)(e_1 + e_2 + e_3)+(v-1)(\ell_{1,4}+\ell_{2,4})&\mbox{ for } 1\le v \le \frac{u+1}{2}.
        \end{cases}
    \end{align*}}
    \item if $u\in [2,3]$ then:
     {\allowdisplaybreaks\begin{align*}
        &N(u,v)=
        \begin{cases}
        0&\mbox{ for } 0\le v \le 1,\\
        (v-1)(\ell_{1,4}+\ell_{2,4})&\mbox{ for } 1\le v\le u-1,\\
        (-u + v + 1)(e_1 + e_2 + e_3)+(v-1)(\ell_{1,4}+\ell_{2,4})&\mbox{ for } u-1 \le v \le \frac{u+1}{2}.
        \end{cases}
    \end{align*}}
\end{itemize}
We have
{\allowdisplaybreaks\begin{align*}
    S\big(&W_{\bullet,\bullet}^{\widetilde F};Z\big)\leqslant \frac{3}{22}\int_{0}^{3}\int_{0}^{\infty}\mathrm{vol}\Big(P(u)\big\vert_{\widetilde F}-v\ell_{123}\Big)dvdu=\\
    &=\frac{3}{22}\Bigg(\int_0^{1}\int_0^u 2(u - v)^2 dvdu + \\
    &\text{ }\text{ }\text{ }\text{ }\text{ }\text{ }+\int_1^{2}\Big(\int_0^{u-1}  (- 2u^2 + 2uv - v^2 + 8u - 6v - 4)dv + \int_{u-1}^1  (u^2 - 4uv + 2v^2 + 2u - 1) dv +\\
    &\text{ }\text{ }\text{ }\text{ }\text{ }\text{ }\text{ }\text{ }\text{ }\text{ }\text{ }\text{ }+ \int_{1}^{\frac{u+1}{2}}   (1 + u - 2v)^2 dv \Big)du+\\
    &\text{ }\text{ }\text{ }\text{ }\text{ }\text{ }+\int_2^{3}\Big(\int_0^{1} (- 2u^2 + 2uv - v^2 + 8u - 6v - 4)dv + \int_{1}^{u-1}  (-2u^2 + 2uv + v^2 + 8u - 10v - 2)dv +\\
    &\text{ }\text{ }\text{ }\text{ }\text{ }\text{ }\text{ }\text{ }\text{ }\text{ }\text{ }\text{ }+ \int_{u-1}^{\frac{u+1}{2}}   (1 + u - 2v)^2 dv\Big)du\Bigg)=\frac{29}{44}<1.
\end{align*}}
We see that $S_X(\widetilde F)<2$ and $S\big(W^{\widetilde F}_{\bullet,\bullet};Z\big)<1$ thus
\begin{equation*}
\frac{A_X(\Xi)}{S_X(\Xi)}\geqslant\min\Bigg\{\frac{2}{S_X(\widetilde F)},\frac{1}{S\big(W^{\widetilde F}_{\bullet,\bullet};Z\big)}\Bigg\}>1.
\end{equation*}
and $\beta(\Xi)=A_X(\Xi)-S_X(\Xi)> 0$.
\end{proof}
Finally, we exclude the case where $Z\subset \Gamma_1$. 
\begin{lemma}

\label{lemma:I-EC}
If the center $Z= c_{X_1}(\Xi)$ is one-dimensional and is contained in $E_1$, $\beta(\Xi)>0$. 
\end{lemma}

\begin{proof}
Assume that $Z\subset E_1$, then since there is no $G$-fixed point on $\DP^3$, $\phi_1(Z)= c_{\DP^3}(\Xi)$ is the curve $\Gamma_1$. 
Denote by $Y_1\to \DP^3$ the blowup of the line $L_1$ and by $Y_1\to \DP^1$ the morphism induced by the projection $\DP^3 \dashrightarrow \DP^1$ away from $L_1$. Let $\widehat X^+\to Y_1$ be the blowup of the proper transform of $\Gamma_1$, then $\widehat X^+\dashrightarrow \widehat X$ is a flop, and there is a morphism $\widetilde X\to \widehat X$. Denote by $\eta$ the composition $\widetilde X\to \widehat X\to Y_1\to \DP^1$ and by $\widetilde Z$ the centre $c_{\widetilde X}(\Xi)$. 
If $T$ is a~general fiber of $\eta$, $T\cdot \widetilde Z\geqslant 5$, hence, by Lemma~\ref{CorollaryA15}, $\beta(\Xi)>0$.
\end{proof}

\begin{lemma}\label{I:dim2}
    There is no $G$-invariant prime divisor $\Xi$ over $X$ with centre a prime divisor $D_X= c_X(\Xi)$ such that $\beta(\Xi)\le 0$. 
\end{lemma}
\begin{proof} 

By \cite[Corollary 1.44]{Fano21}, for any divisor $\Xi$ over $X$, if $\alpha_{G,Z}(X)>3/4$, where $Z=c_X(\Xi)$, then $\beta(\Xi)>0$. Assume now that there is a divisor $\Xi$ over $X$ with $\beta(\Xi)\le 0$ and $c_{X}(\Xi)= D_X$ a divisor, so that $\alpha_{G,D_X}(X)\leq \frac{3}{4}$. First assume that $\alpha_{G,D_X}(X)< \frac{3}{4}$, then $D_X$ is the $G$-orbit of a minimal log canonical centre of a suitable pair $\big(X, \frac{3}{4} \mathcal D)$ for $\mathcal D\subset |-K_X|_{\DQ}$ a $G$-invariant linear system. By \cite[Theorem 1.52]{Fano21}, $D_X$ is a $G$-invariant irreducible normal surface with 
\begin{equation*}\label{eqonX}
    -K_X \sim_{\DQ} \lambda D_X+ \Delta_X
\end{equation*} 
for $\Delta_X$ an effective $\DQ$-divisor and a rational number $\lambda> \frac{4}{3}$.

We show that there is no such divisor $D_X$. Recall that $\pi\colon X_i\to X$ for $i=1,2$ are small $\DQ$-factorialisations so that 
\[ -K_{X_i}\sim _{\DQ} \lambda D_X+ \Delta_X\]
where we still denote by $D_X, \Delta_X$ the pullbacks of these divisors to $X_i$. 
We have $\Effb(X_i)= \DR_{\ge 0}[E_1]+ \DR_{\ge 0}[E_2]$, and $E_2= 8H_1-3E_1$. If $D_X= E_1$, then $\Delta_X\sim 4H_1-(1+\lambda)E_1$, but this is impossible as $(1+\lambda)> 3/2$. If $D_X\neq E_1$, the image of $D_X$ by $f_1$ is a $G$-invariant irreducible surface of degree $d\in \DN$ on $\DP^3$, and since 
\[ f_1(4H_1) \sim \lambda f_1(D_X) + f_1(\Delta_X)\]
we have $4 \geq \lambda d$, and since $\lambda>\frac{4}{3}$, $d\leq 2$. Since there is no $G$-invariant hyperplane of $\DP^3$, $d=2$ and $f_1(D_X)$ is a $G$-invariant quadric. Since $\Gamma_1$ doesn't lie on a quadric, $D_X\sim 2H_1$ and
\[ \Delta_X\sim (4-2\lambda)H_1-E_1 \sim xH_1-E_1\]
for some $x<4/3$, which is impossible for an effective divisor.

Now assume that $\alpha_{G,D_X}(X)= \frac{3}{4}$, then since $\pi_1$ is small, $\alpha_{G,D_X}(X_1)= \frac{3}{4}$ by \cite[Lemma 1.47]{Fano21}, and $\beta(\Xi)= A_{X}(D_X)- S_X(D_X)= A_{X_1}(D_X)-S_{X_1}(D_X)$. Since $X_1$ is smooth, as in the proof of \cite[Theorem 1.51]{Fano21}, assuming that $\beta(\Xi)=0$ would imply $X_1\simeq \DP^3$, a contradiction. We conclude that $\beta(\Xi)>0$ for all $G$-invariant prime divisors $\Xi$ with $c_X(\Xi)=D_X$ a prime divisor on $X$. 
\end{proof}

We now have all the elements to prove
\begin{maintheorem1*}
The threefold $X$ is $K$-polystable.
\end{maintheorem1*}
\begin{proof}
    Assume that $X$ is not $K$-polystable, then there is a $G$-invariant prime divisor $\Xi$ over $X$ such that $\beta(\Xi)<0$. 
    Lemma~\ref{I:dim2} shows that the centre of $\Xi$ on $X$ is not a surface. 
    If the centre of $\Xi$ on $\DP^3$ is a curve other than $\Gamma_1$, by Lemma~\ref{CorollaryA13}, this curve is a line. Lemma~\ref{dim1a} shows that this line cannot be a $G$-invariant line that is not the unique quadrisecant of $\Gamma_1$, while Lemma~\ref{dim1b} excludes the quadrisecant line $L_1$. Lemma~\ref{lemma:I-EC} shows that the centre of $\Xi$ on $\DP^3$ is not $\Gamma_1$. 
    As there is no $G$-fixed point on $\DP^3$, if $c_X(\Xi)$ is $0$-dimensional, it is the singular point $x_0\in X$, and its centre on $\DP^3$ is $L_1$, so that this case is also excluded by Lemma~\ref{dim1b}.
\end{proof}
Since $\Aut (X)$ is finite, $X$ is K-stable, and by openness of K-stability \cite{BL22}, this implies: 
\begin{corollary}
  A general one-nodal prime Fano threefold of genus $12$ in Family I is K-stable.   
\end{corollary}
\begin{remark}
    Liu and Zhao have constructed a K-semistable degeneration of one-nodal prime Fano threefolds in Family I, in which the curve $\Gamma_1$ is taken to lie on a quadric (this corresponds to $C_{a,b}$ with $|a|=|b|$ above). The resulting prime Fano threefold of genus $12$ has (non-isolated) canonical singularities \cite{LiuZhao}. 
\end{remark}

\section{Family II}
Let $X$ be a one-nodal prime Fano threefold of genus $12$ that belongs to Family II of Theorem~\ref{fams}, then $X$ is the midpoint of a Sarkisov link associated 
to a rational map $Q\subset \DP^4\dashrightarrow \DP^2$; we describe the associated birational geometry briefly, see \cite{JRP11, Pro16} and \cite{KP23} for proofs and precise statements. 

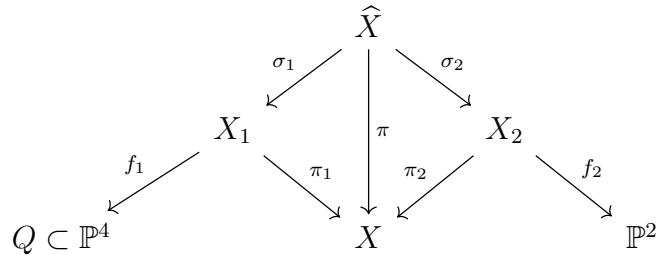
\begin{figure}[h!]
    \centering
    \begin{tikzcd}&& \widehat X \arrow[dl,  "\sigma_1"'] \arrow[dr,  "\sigma_2"] \arrow[dd,  "\pi"]  &&\\
    & X_1  \arrow[dl,  "f_1"'] \arrow[dr,  "\pi_1"]  && X_2 \arrow[dl,  "\pi_2"'] \arrow[dr,  "f_2"]\\
   Q\subset \DP^4 && X && \DP^2
    \end{tikzcd}
\end{figure}
Denote by $H_1= \sigma_1^*\big( f_1^*\mathcal{O}_{Q}(1)\big)$, by $H_2= \sigma_2^*\big(f_2^* \mathcal O_{\DP^2}(1)\big)$, and by $H= \pi^*(-K_X)$ the pullbacks to $\widehat X$ (or to any of the models) of the ample generators of $\Pic (Q)$, $\Pic (\DP^2)$ and $\Pic (X)$ respectively. The morphism $f_1$ is the blowup of a smooth rational quintic curve $\Gamma_1\subset Q\subset \DP^4$ that does not lie on a hyperplane section of $Q$ ($|H_1-\Gamma_1|=\emptyset$), and there is a unique trisecant line $L_1$ to $\Gamma_1$. 
The curve $\Gamma_1$ lies on a section of $|2H_1|$ on $Q$, that is on a del Pezzo surface of degree $4$, and the linear system $|2H_1-\Gamma_1|$ has dimension $3$ and $\Bs |2H_1-\Gamma_1|= \Gamma_1\cup L_1$. The rational map associated to $|2H_1-\Gamma_1| = |H_2|$ is precisely the $Q\dashrightarrow \DP^2$ induced by the Sarkisov link above. The threefold $X_1$ is weak Fano, 
\begin{equation*}\label{eq2.1}
    -K_{X_1}\sim H\sim 3H_1-E_1
\end{equation*}
where $E_1= \Exc f_1$, 
so that the proper transform of $L_1$ (still denoted $L_1$) is the unique flopping curve on $X_1$. The map $\pi_1$ contracts $L_1$ to a node $\{x_0\}= \operatorname{Sing}(X) \in X$.

Let $\pi\colon \widehat X\to X$ be the blowup of $x_0$, and $\sigma_1$ the induced map to $X_1$, note that $L_1= \sigma_1( F)$, where $ F= \Exc \pi$. The threefolds $X_1$ and $X_2$ are the two small resolutions of the node $x_0$ and  $\chi$ is the induced birational map between these (the Atiyah flop associated to $x_0\in X$).
Then, $\widehat X$ is a weak Fano threefold ($-K_{\widehat X}$ is nef and big) with $\rho=3$ and we have \cite{KP23}:
\begin{equation*}\label{eq2.2}
 -K_{\widehat X}\sim  H-F \sim H_1+H_2   
\end{equation*}
and from $H\sim 3H_1-E_1 \sim H_1+H_2$, we deduce
\begin{equation*}\label{eq2.3}
    H_2\sim 2H_1-E_1.
\end{equation*}
The map $f_2$ is a conic bundle (a Mori fibre space with one-dimensional fibres) and its discriminant curve $\delta= -{f_2}_*(K_{X_2/\DP^2})^2= -{f_2}_*(H-3H_2)^2$
has degree $12$. 

Denoting by $\mathcal E= {\phi_2}_*H$, then $X_2\subset \DP(\mathcal E)$ is a section of $2H-3H_2$, where, abusing notation, we denote by $H$ the tautological class of $\DP(\mathcal{E})$ and by $H_2$ the pullback of $\mathcal O_{\DP^2}(1)$). 

Since $H_2\cdot L_1=1$, $L_1$ maps to a line $\ell\subset \DP^2$. Let $R= f_2^{-1} \ell$ be its preimage on $X_2$, and by abuse of notation, also denote by $R= \sigma_2^*(f_2^{-1}\ell)$ its proper transform on $\widehat X$. By construction, $R$ is the unique section of $|2H_1-E_1-2F|= |H_2-F|$. 

\subsection{Construction of a member with $\DZ_2\rtimes \DZ_3$-action} Let $Q\subset \DP^4$ be the smooth quadric threefold 
\[Q= \{ 2 x_2^2=x_1x_3-x_0x_4\} \]
and let $\Gamma_1$ be the image of the embedding $\DP^1\hookrightarrow \DP^4$ given by 
$$[x:y]\to [x^5: 2x^3y^2+y^5:x^4y+xy^4:2x^2y^3+x^5 :y^5];$$
$\Gamma_1$ lies on $Q$ but on no hyperplane section of $Q$. 

Let $\omega$ be a primitive cube root of unity, and define an action of $G:=\DZ/2\DZ\rtimes\DZ/3\DZ$  on $\DP^4$ by the action of its generators:
\begin{align*}
    &\tau : [x_0:x_1:x_2:x_3:x_4]\to [x_4:x_3:x_2:x_1:x_0],\\
    &\sigma : [x_0:x_1:x_2:x_3:x_4]\to [x_0:\omega^2x_1:\omega x_2:x_3:\omega^2 x_4 ].
\end{align*}
and observe that $\Gamma_1$ is $G$-invariant, and that 
\[ L_1= \{ x_0+x_3= x_4+x_1= x_2=0\}\]
is $G$-invariant and trisecant to $\Gamma_1$ (the intersection $L_1\cap \Gamma_1$ consists of the image of the points $[1:-\omega^i]$ for $i=0,1,2$). 
The threefolds $X,X_1$ and $X_2$ are equipped with a $G$-action and the Sarkisov link above is $G$-equivariant. For instance, a $G$-invariant basis of $|H_2|$ is  
\[ \begin{cases}
    S_1=\{ x_0^2-x_3^2+2(x_1-x_4)=0\},\\
    S_2=\{ x_1^2-x_4^2+2(x_0-x_3)=0\},\\
    S_3=\{(x_0-x_2)^2+(x_4-x_2)^2+x_0x_1 + 2x_0x_4  + x_3x_4= 2x_2^2+(x_1-x_2)^2+(x_3-x_2)^2\}.
\end{cases}\]
and the only section of $|H_2|$ that is singular along $L_1$ is $f_1\big(\sigma_1 R\big)$ ( which we still call $R$ by abuse of notation). We have:
\begin{equation*}
    R= \{ x_0x_1+2x_0x_4+x_3x_4= 2x_2^2\}. 
\end{equation*}

The discriminant curve of $f_2$ is the smooth plane cubic 
\[ \delta=\{ 2y_0^3 + 6y_0^2y_1 + 5y_0y_1^2 + y_0y_1y_2 + 3y_1^3 +5y_1^2y_2 + 6y_1y_2^2 +2y_2^3=0\}\subset \DP^2.\]
\begin{claim}
    The group $\Aut(X)$ is finite.
\end{claim}
\begin{proof}
 Since $\Aut(X)$ is a subgroup of $\Aut(X_1)= \Aut(Q, \Gamma_1)$, and since $\Gamma_1$ does not lie on a hyperplane section of $Q$, $\Aut(Q, \Gamma_1)= \Aut(X_1)$ is a subgroup of $\Aut (\Gamma_1)= \Aut (\DP^1)$ by \cite[Lemma 2.1]{CPS19}. Consequently,  $\Aut(X)$ is a subgroup of $\Aut (\DP^1)$ preserving the three points of intersection $\Gamma_1\cap L_1$, therefore it is finite.  
\end{proof}

The intersection numbers associated to the Sarkisov link are 
\begin{align*}
     &H_1^3=2,& &H_1^2\cdot E_1 = 0, & &H_1\cdot E_1^2=-5, & &E_1^3=-13,\\
    &      & &H_1^2\cdot F = 0, & &H_1\cdot F^2=-1, & &F^3=2,\\
    &      & &E_1\cdot F\cdot H_1=0,& &E_1\cdot F^2=-3, & &E_1^2\cdot F=0,\\
    &H_2^3=0,& &H_2^2\cdot H=2, & &H_2\cdot H^2=12-\deg \delta.& &
\end{align*}

We will apply Theorem~\ref{equivGstab} to prove that $X$ is K-stable. To do so, we first describe possible centres of $G$-invariant divisors over $X$. In what follows, $\Xi$ always denotes a $G$-invariant prime divisor over $X$.

\begin{claim}\label{cla2.1}
If the centre of $\Xi$ on $X$ is $0$-dimensional, it is the singular point $c_X(\Xi)= \{x_0\}$.     
\end{claim}
    
\begin{proof}
There is no point of $Q\subset \DP^4$ fixed by the action of $G$.
\end{proof}

We now consider the case when the centre $Z=c_{Q}(\Xi)$ on $Q$ is one-dimensional. 
First, we assume that $Z$ lies on a (smooth) section $S$ of the linear system $|H_2|=|2H_1-E_1|$. 
As an intersection of two quadrics in $\DP^4$, $S$ is a del Pezzo of degree $4$, and $p\colon S\to \DP^2$ is the blowup of five points $p_1, \cdots , p_5$ in general position. Let $\ell$ be the pullback of a line on $\DP^2$, and $e_1, \cdots , e_5$ the $p$-exceptional curves. Then the Mori cone $\NEb(S)$ is generated by $\ell, e_1, \cdots , e_5, \ell_{i,j}$ for $1\leq i<j\leq 5$ and $q$ where $\ell_{i,j}$ is the proper transform of the line through $p_i$ and $p_j$ and $q$ that of the conic through $p_1, \cdots, p_5$. For a smooth curve $C\subset S$, if $C\sim k\ell+ \sum  m_ie_i$, then 
\[ \deg C= -K_S\cdot C= H_1\cdot C= 3k-\sum m_i \mbox{ and } p_a(C)= \frac{(k-1)(k-2)}{2}- \sum \frac{m_i(m_i-1)}{2}\]
so that without loss of generality, we may assume that $\Gamma_1= 2\ell-e_1$ and $L_1= q$. 

\begin{lemma}\label{II.1a}
    If $Z= c_{Q}(\Xi)$ is a $G$-invariant irreducible curve lying on $S\in |H_2|$, and if $Z\not \subset \Gamma_1\cup L_1$, then $\beta(\Xi)>0$. 
\end{lemma}
\begin{proof}
    We use Theorem~\ref{corollary:Kento-formula-Fano-threefold-surface-curve} to bound $\beta(\Xi)$ below. Let $D_u= H-uS$ on $\widehat X$ for $u\geq 0$, and write its Zariski decomposition $D_u= P(u)+N(u)$, where for $0\le u\le \frac{3}{2}$, $P(u)$ is nef and 
\[
P(u)=H-uS-u\Big(\frac{E_1}{3}+F\Big)= \Big(1-\frac{2}{3}u\Big)(3H_1-E_1)\mbox{ and }
    N(u)=u\Big(\frac{E_1}{3}+F \Big),
\]
which gives:
\begin{align*}
   S_X(S)&=\frac{1}{(-K_X)^3}\int_{0}^{\tau}\mathrm{vol}(\pi^*(-K_X)-uS)du= \frac{1}{22}\int_{0}^{3/2}\frac{22(3 - 2u)^3}{27} du=\frac{3}{8}<1.
\end{align*}
Note that since $Z\not \subset \big(E_1\cup F\big)$, $\operatorname{ord}_Z \big(N(u)\vert_S\big)=0$.

We now consider $(H-uS)_{\vert S}-vZ$ on $S$ and denote by $P(u,v)+N(u,v)$ its Zariski decomposition for $0\leq u \leq \frac{3}{2}$.  We have 
$$Z\sim \alpha\ell+\sum_{i=1}^5 \alpha_i e_i+ \sum_{1\leq i<j\leq 5} \alpha_{ij}\ell_{ij} + \beta q.$$
Since $Z\not \subset F$, $Z\ne q$ and at least one of the coefficients $\alpha, \alpha_i, \alpha_{i,j}$ is $\geq 1$. 
If $\mathbf l$ is the corresponding curve, since $Z\geq \mathbf{l}$, by convexity of volume:
$$
S\big(W_{\bullet,\bullet}^{S};Z\big)=\frac{3}{22}\int_{0}^{3}\int_{0}^{\infty}\mathrm{vol}\Big(P(u)\big\vert_{\widetilde F}-vZ\Big)dvdu\leqslant \frac{3}{22}\int_{0}^{3}\int_{0}^{\infty}\mathrm{vol}\Big(P(u)\big\vert_{\widetilde F}-v\mathbf{l}\Big)dvdu.
$$
so it is enough to show that the last integral is less than $1$ when $Z=\mathbf l$, for each possible $\mathbf l$.

\par {\bf Case 1.} $Z\sim \ell$.  For $0\leq u \leq \frac{3}{2}$, and $0\leq v \leq \frac{3-2u}{3}$, $N(u,v)=vq$, and we compute
{\allowdisplaybreaks\begin{align*}
    S\big(W_{\bullet,\bullet}^{S};Z\big)&\leqslant \frac{3}{22}\int_{0}^{3/2}\int_{0}^{\infty}\mathrm{vol}\Big(P(u)\big\vert_{S}-v\ell\Big)dvdu=\\
    &=\frac{3}{22}\int_0^{3/2}\int_0^{1-\frac{2u}{3}} \frac{(2u + 3v - 3)(6u + 5v - 9)}{3} dvdu =\frac{3}{16}<1.
\end{align*}}

\par{\bf Case 2.}  $Z\sim e_1$. For $0\leq u \leq \frac{3}{2}$, we have:
\[ N(u,v)= \begin{cases} vq & \text{ for }  0\leq v\leq \frac{2(3-2u)}{3},\\
        vq+(v-2 + \frac{4u}{3} )(\ell_{12} + \ell_{13} + \ell_{14} + \ell_{15})& \text{ for } \frac{2(3-2u)}{3}\leq v\leq \frac{5(3-2u)}{6}. 
\end{cases}\]
We obtain {\allowdisplaybreaks\begin{align*}
    S\big(&W_{\bullet,\bullet}^{S};Z\big)\leqslant \frac{3}{22}\int_{0}^{3/2}\int_{0}^{\infty}\mathrm{vol}\Big(P(u)\big\vert_{S}-ve_1\Big)dvdu=\\
    &=\frac{3}{22}\int_0^{3/2}\Big(\int_0^{\frac{2(3-2u)}{3}} \frac{(2u - 3)(6u + 4v - 9)}{3} dv + \int_{\frac{2(3-2u)}{3}}^{\frac{5(3-2u)}{6}}{\frac{5(3-2u)}{6}} \frac{(10u + 6v - 15)^2}{9} dv\Big)du =\\
    &=\frac{182}{352}<1.
\end{align*}}
\par {\bf Case 3.} $Z\sim e_2$ (or $e_i$, $i\neq 1$). For $0\leq u \leq \frac{3}{2}$, we have:
\[ N(u,v)= \begin{cases} vq & \text{ for }  0\leq v\leq \frac{3-2u}{3},\\
vq+(v-1 + \frac{2u}{3} )(\ell_{23} + \ell_{24} + \ell_{25})& \text{ for } \frac{3-2u}{3}\leq v\leq \frac{2(3-2u)}{3}. \end{cases}\]
In addition
{\allowdisplaybreaks\begin{align*}
    S\big(&W_{\bullet,\bullet}^{S};Z\big)\leqslant \frac{3}{22}\int_{0}^{3/2}\int_{0}^{\infty}\mathrm{vol}\Big(P(u)\big\vert_{S}-ve_2\Big)dvdu=\\
    &=\frac{3}{22}\int_0^{3/2}\Big(\int_0^{\frac{3-2u}{3}} (2u - 3)(2u + 2v - 3) dv + \int_{\frac{3-2u}{3}}^{\frac{2(3-2u)}{3}}  \frac{(4u + 3v - 6)^2}{3} dv\Big)du =\\
    &=\frac{63}{176}<1.
\end{align*}}

\par {\bf Case 4.} $Z\sim \ell_{12}$ (or $\ell_{1j}$). For $0\leq u \leq \frac{3}{2}$, we have: 
\[N(u,v)=\begin{cases}
        0 & \text{ for }  0\leq v\leq \frac{3-2u}{3},\\
        (v-1 + \frac{2u}{3} )(\ell_{34} + \ell_{35} + \ell_{45})& \text{ for } \frac{3-2u}{3}\leq v\leq \frac{2(3-2u)}{3}.
    \end{cases} \]
In addition 
{\allowdisplaybreaks\begin{align*}
    S&\big(W_{\bullet,\bullet}^{S};Z\big)\leqslant \frac{3}{22}\int_{0}^{3/2}\int_{0}^{\infty}\mathrm{vol}\Big(P(u)\big\vert_{S}-v\ell_{12}\Big)dvdu=\\
    &=\frac{3}{22}\int_0^{3/2}\Big(\int_0^{\frac{3-2u}{3}} 4u^2 + \frac{8}{3}uv-12u -9  - 4v - v^2 dv +\int_{\frac{3-2u}{3}}^{\frac{2(3-2u)}{3}}  \frac{2(4u + 3v - 6)(2u + v - 3)}{3} dv\Big)du =\\
    &=\frac{75}{176}<1.
\end{align*}}
 \par {\bf Case 5.} $Z\sim \ell_{23}$ (or $\ell_{ij}$, $i\neq 1$). For $0\leq u \leq \frac{3}{2}$, we have: 
\[N(u,v)=\begin{cases}
        0 & \text{ for } 0\leq v\leq \frac{3-2u}{3},\\
        (v-1 + \frac{2u}{3} )\ell_{45}&\text{ for } \frac{3-2u}{3}\leq v \leq \frac{2(3-2u)}{3},\\
        (v-1 + \frac{2u}{3} )\ell_{45} -(v-2 + \frac{4u}{3} )(\ell_{14} + \ell_{15})&\text{ for }\frac{2(3-2u)}{3}\leq v \leq 3-2u.
    \end{cases}\]
    In addition
{\allowdisplaybreaks\begin{align*}
    S\big(W_{\bullet,\bullet}^{S};Z\big)&\leqslant \frac{3}{22}\int_{0}^{3/2}\int_{0}^{\infty}\mathrm{vol}\Big(P(u)\big\vert_{S}-v\ell_{23}\Big)dvdu=\\
    &=\frac{3}{22}\int_0^{3/2}\Big(\int_0^{\frac{3-2u}{3}} 4u^2  +\frac{4}{3}uv- 12u + 9 - 2v   - v^2 dv + \\
    &\text{ }\text{ }\text{ }\text{ }\text{ }\text{ }\text{ }\text{ }\text{ }\text{ }\text{ }\text{ }+\int_{\frac{3-2u}{3}}^{\frac{2(3-2u)}{3}}  \frac{2(2u - 3)(10u + 6v - 15)}{3} dv + \int_{\frac{2(3-2u)}{3}}^{3-2u}  2(2u + v - 3)^2 dv\Big)du =\\
    &=\frac{111}{176}<1.
\end{align*}}
This finishes the proof, as in all cases we have $\min\Big\{ \frac{1}{S_X(S)}, \frac{1}{S\big(W_{\bullet,\bullet}^{S};Z\big)}\Big\}>1$
\end{proof}
\begin{lemma}\label{II.1b}
   If $Z= c_{Q}(\Xi)$ is a line other than $L_1$, $\beta(\Xi)>0$.
\end{lemma}
\begin{proof}
    Since there is no $G$-fixed point on $Q$, $Z\cap \Gamma_1$ is empty or consists of two points. In the second case, $H_2\cdot Z=0$, so that $Z$ lies on a section $S\in |H_2|$ and $\beta(\Xi)>0$ by Lemma~\ref{II.1a}.

We now assume that $Z$ is disjoint from $\Gamma_1$ and  denote by $S^Q\simeq \DP^1\times \DP^1$ the general hyperplane section of $Q$ containing $Z$, and by $S$ its proper transform on $\widehat X$. The intersection $S^Q\cap \big(\Gamma_1\cup L_1\big)= \{ p_1, \cdots , p_6\}$ consists of six points, and these points are in general position because any line through $3$ of the points (respectively conic through $6$ of the points) would be contracted by $\pi_1$, the anticanonical map of $X_1$, but the only flopping curve on $X_1$ is $L_1$. As $S$ is the blowup of $S^Q$ at $\{p_1, \cdots , p_6\}$, $S$ a del Pezzo surface of degree $2$. Denote by  $\ell_1, \ell_2$ the pullbacks of the two rulings of $S^Q= \DP^1\times \DP^1$, and by $e_1, \cdots, e_6$ the exceptional divisors. The Mori cone $\NEb(S)$ is generated by $\ell_1, \ell_2, e_1, \cdots, e_6$, and by the classes of
    \begin{enumerate}[-]\item the proper transforms $\ell_{i(1)}$ and $\ell_{i(2)}$ of rulings through the points $p_i$ for $1\le i\le 6$, 
    \item the proper transforms $\ell_{i,j,k}$ for $1\le i<j<k\le 6$ of irreducible conics through 3 of the blownup points ($\ell_{i,j,k}= \ell_1+\ell_2-e_i-e_j-e_k$),
    \item the proper transforms $\kappa_{j(1)}$ and $\kappa_{j(2)}$ of rational cubic curves though $5$ of the $p_i$s (where $\kappa_{j(1)}= 2\ell_1+\ell_2-\sum e_i+e_j$) for $1\le j\le 6$,
    \item and the proper transforms $q_j$ of elliptic quartic curves through $p_1, \cdots, p_6$, which have multiplicity $2$ at $p_j$ for $1\le j\le 6$ ($q_j= 2\ell_1+ 2\ell_2- \sum e_i-e_j$).
    \end{enumerate}
The Zariski decomposition of $\pi^*(-K_X)-uS$ writes $P(u)+N(u)$ where $P(u)$ is nef, and for $0\leq u \leq 1$, $N(u)= uF$. 
We have 
\begin{align*}
   S_X(S)&=\frac{1}{(-K_X)^3}\int_{0}^{\tau}\mathrm{vol}(\pi^*(-K_X)-uS)du= \frac{1}{22}\int_{0}^{1}(1 - u)(u^2 - 17u + 22) du=\frac{3}{8}<1.
\end{align*}

Since $Z$ is disjoint from $\Gamma_1$, without loss of generality we may assume that $Z\sim \ell_1$. 

Using the same notation as before, for $0\leq u\leq \frac{5}{7}$, we have
\[ N(u,v)= \begin{cases}
    0 &\mbox{ for } 0\le v \le \frac{4-3u}{2},\\
     (2v-4 + 3u)\kappa_{6(2)} &\mbox{ if } \frac{4-3u}{2}\leq v\leq \frac{5-4u}{2},\\
        (2v-5 + 4u)\sum \kappa_{i(2)}+(1-u)\kappa_{6(2)} &\mbox{ if }\frac{5-4u}{2}\leq v\leq  \frac{7-5u}{3}.
    \end{cases}\]
For $\frac{5}{7}\le u \le 1$, we have
\[ N(u,v)= \begin{cases}
    0 &\mbox{ for } 0\le v \le \frac{4-3u}{2},\\
     (2v-4 + 3u)\kappa_{6(2)} &\mbox{ if } \frac{4-3u}{2}\leq v\leq \frac{5-4u}{2},\\
        (2v-5 + 4u)\sum \kappa_{i(2)}+(1-u)\kappa_{6(2)} &\mbox{ if }\frac{5-4u}{2}\leq v\leq  \frac{11-9u}{4}.
    \end{cases}\]
Since $Z\not \subset F$, $\ord_{Z}\Big(N(u)\big\vert_S\Big)=0$ and 
{\allowdisplaybreaks\begin{align*}
    S\big(&W_{\bullet,\bullet}^{S};Z\big)= \frac{3}{22}\int_{0}^{1}\int_{0}^{\infty}\mathrm{vol}\Big(P(u)\big\vert_{S}-vZ\Big)dvdu=\\
    &=\frac{3}{22}\int_0^{5/7}\Big(\int_0^{\frac{4-2u}{2}}  u^2 + 2uv - 12u - 6v + 13 dv + \int_{\frac{4-3u}{2}}^{\frac{5-4u}{2}}  10u^2 + 14uv - 36u + 4v^2  - 22v + 29 dv +\\
    &\text{ }\text{ }\text{ }\text{ }\text{ }\text{ }\text{ }\text{ }\text{ }\text{ }\text{ }\text{ }+ \int_{\frac{5-4u}{2}}^{\frac{7-5u}{3}}  2(5u + 3v - 7)(9u + 4v - 11)dv\Big)du +\\
    &+\frac{3}{22}\int_{5/7}^{1}\Big(\int_0^{\frac{4-2u}{2}}  u^2 + 2uv - 12u - 6v + 13 dv + \int_{\frac{4-3u}{2}}^{\frac{5-4u}{2}}  10u^2 + 14uv - 36u + 4v^2  - 22v + 29 dv +\\
    &\text{ }\text{ }\text{ }\text{ }\text{ }\text{ }\text{ }\text{ }\text{ }\text{ }\text{ }\text{ }+ \int_{\frac{5-4u}{2}}^{\frac{11-9u}{4}}  2(5u + 3v - 7)(9u + 4v - 11)dv\Big)du=\frac{18969}{1108811}<1.
\end{align*}}
As above, this completes proof that $\beta(\Xi)>0$.
\end{proof}
\begin{lemma}\label{II.1c}
    If $c_{Q}(\Xi)= L_1$, then $\beta(\Xi)>0$. 
\end{lemma}
\begin{proof}

  By \cite{KP23}, there are precisely 3 lines through $x_0\in X$, and by construction, the set of lines through $x_0\in X$ is $G$-invariant. If $L\ni x_0$ is a line and $\widehat{L}$ is its proper transform on $\widehat X$, $-K_{\widetilde X}\cdot \widehat L=0$ and $\widehat L$ is a flopping curve.  Let $\omega\colon \widetilde{X}\to \widehat X$ be the blowup of the proper transforms of the 3 lines through $x_0\in X$, and denote by $\Lambda= \Lambda_1+ \Lambda_2+\Lambda_3$ its ($G$-invariant) exceptional divisor. Denote by $\widetilde F= \omega^* F$, $\widetilde E_1= \omega^* E_1-\Lambda$ the proper transforms of $F$ and $E_1$ on $\widetilde{X}$, and by $\widetilde R = \omega^* R-\Lambda$, the proper transform of the unique section of $|2H_1-E_1-2F|= |H_2-F|$. On $\widetilde X$, we have the intersection numbers: \begin{align*}
  &\Lambda^3=6, & & \Lambda^2\cdot \omega^* F = -3, & & \Lambda\cdot \omega^*(F)^2 = 0, & & \Lambda \cdot \omega^*(F)\cdot \omega^*(E_1) =0 ,\\
    &  \Lambda^2 \cdot  \omega^*(E_1) = 3, & & \Lambda \cdot  \omega^*(E_1)^2 =0, & & \omega^*F^3= 2,&& \omega^*E_1^3= -13.\\
    &      & &
\end{align*}

We first show that $\beta(F)>0$. 
The Zariski decomposition of $\omega^*\pi^* (-K_X)- u\widetilde F$ can be written $P(u)+N(u)$, where $P(u)$ is nef and
 \[ N(u)=
    \begin{cases}
   0 & \mbox{ for } 0 \le u \le 1,\\
    (u-1)\Lambda & \mbox{ for } 1\le u\le 2, \\
    (u-1)\Lambda+ (u-2)\widetilde{R} &\mbox{ for } 2\le u\le 3.
    \end{cases}
\]

   
We have $A_X(\widetilde F)=2$ and
\begin{align*} \label{eq:EL-II}
   S_X(\widetilde F)&=\frac{1}{(-K_X)^3}\int_{0}^{\tau}\mathrm{vol}(\omega^*\pi^*(-K_X)-u\widetilde F)du\\
   &=\frac{1}{22}\Bigg(\int_{0}^{1} 22-2u^3 du + \int_{1}^{2} (u + 1)(u^2 - 10u + 19) du  + \int_{2}^{3}   3(u - 3)(2u - 7) du  \Bigg)=\frac{161}{88}.
\end{align*}

So that $\beta(F)=\frac{15}{88}>0$.

Now assume that $\Xi$ is not $F$ and denote by $Z$ the centre of $\Xi$ on $\widetilde X$. By construction, $Z= c_{\widetilde X}(\Xi) \subset \widetilde F$  is a curve, and $\widetilde F$ is a blowup of $\DP^1\times \DP^1$ in three points in general position, so it is a del Pezzo surface of degree $5$. We denote by $\ell_1, \ell_2$ the proper transforms of the two rulings of $\DP^1\times \DP^1$, and by $e_1, e_2, e_3$ the $(-1)$-curves 
The extremal rays of the Mori cone $\NEb(\widetilde F)$ are the $(-1)$-curves $e_1, e_2, e_3$, the proper transforms $\ell_{i(1)}$ and $\ell_{i(2)}$ of rulings through the blownup points for $1\le i\le 3$, and the proper transform of the conic through the three blownup points $\ell_{123}= \ell_1+\ell_2-e_1-e_2-e_3$. 

 We will estimate $\beta(\Xi)$ by considering the flag $Z\subset \widetilde F\subset \widetilde X$; we write
\begin{multline*}
S\big(W^{\widetilde F}_{\bullet,\bullet};Z\big)=\frac{3}{(-K_X)^3}\int_0^{3}\big(P(u)^{2}\cdot \widetilde F\big)\cdot\ord_{Z}\Big(N(u)\big\vert_{\widetilde F}\Big)du+\\+\frac{3}{(-K_X)^3}\int_0^{3}\int_0^\infty \mathrm{vol}\big(P(u)\big\vert_{\widetilde F}-vZ\big)dvdu.
\end{multline*} 
Since $c_Q (\Xi)$ is one-dimensional, $Z\not \subset \Lambda\big \vert_{\widetilde F}$, and $\ord_{Z}\Big(N(u)\big\vert_{\widetilde F}\Big)=0$ unless $Z= \widetilde R\big \vert_{\widetilde F}$.

We first assume that $Z\neq \widetilde R\big \vert_{\widetilde F}$. There are positive integers $\alpha_i, \alpha_{ij}$ and $\alpha_{123}$ so that 
\[ Z\sim \alpha_1e_1+\alpha_2e_2+ \alpha_3e_3+ \sum \alpha_{ij}\ell_{i(j)}+ \alpha_{123}\ell_{123}.\]
Since $Z\not \subset \Lambda\big \vert_{\widetilde F}$, $\alpha_{ij}$ and $\alpha_{123}$ are not all simultaneously $0$.  Let $\mathbf l$ denote one of the $(-1)$ curves other than $e_1,e_2,e_3$ such that $Z\geq \mathbf l$, then by convexity of volume:
$$
S\big(W_{\bullet,\bullet}^{S};Z\big)=\frac{3}{22}\int_{0}^{3}\int_{0}^{\infty}\mathrm{vol}\Big(P(u)\big\vert_{\widetilde F}-vZ\Big)dvdu\leqslant \frac{3}{22}\int_{0}^{3}\int_{0}^{\infty}\mathrm{vol}\Big(P(u)\big\vert_{\widetilde F}-v\mathbf{l}\Big)dvdu.
$$
so it is enough to show that the last integral is less than $1$ when $Z=\mathbf l$.

 \par {\bf Case 1.} Assume that $Z\sim \ell_{123}$, and let $P(u,v)$ and $N(u,v)$ be the positive and negative parts of the Zariski decomposition of  $(\omega^*\pi^*(-K_X)-u\widetilde F|_{\widetilde F}-vZ$. Then, for $0\le u\le 1$, $N(u,v)=
    v(e_1+e_2+e_3)$ for $0\le v\le u$;
for $1\leq u\leq 2$,
   \[ N(u,v)= \begin{cases}
       0 &\mbox{ for } 0\leq v\leq u-1,\\(v-u+1)(e_1+e_2+e_3) &\mbox{ for } u-1\le v\le u.
   \end{cases} 
   \]
   and for $2\leq u\leq 3$, 
 \[ N(u,v)= \begin{cases}
       0 &\mbox{ for } 0\leq v\leq 1,\\
       (v-1)(e_1+e_2+e_3)&\mbox{ for } 1\le v\le 4-u.
   \end{cases} 
   \]
Putting things together, we get:
{\allowdisplaybreaks\begin{align*}
    S\big(&W_{\bullet,\bullet}^{\widetilde F};Z\big)\leqslant \frac{3}{22}\int_{0}^{3}\int_{0}^{\infty}\mathrm{vol}\Big(P(u)\big\vert_{\widetilde F}-v\ell_{123}\Big)dvdu=\\
    &=\frac{3}{22}\Bigg(\int_0^{1}\int_0^u 2(u - v)^2 dvdu + \\
    &+\int_1^{2}\Big(\int_0^{u-1}  -u^2 + 2uv  + 6u - v^2 - 6v - 3  dv + \int_{u-1}^u  2(u - v)^2 dv \Big)du+\\
    &+\int_2^{3}\Big(\int_0^{1}   2uv - v^2 - 4u - 6v + 13 dv + \int_{1}^{4-u} 2(v - 2)(v + u - 4) \Big)du\Bigg)=\frac{29}{44}<1.
\end{align*}}
and $\beta(\Xi)>0$. 
\par {\bf Case 2.} Now assume that $Z\sim \ell_{1(2)}$ (or any $\ell_{i(j)}$). The positive and negative parts of the Zariski decomposition of  $(\omega^*\pi^*(-K_X)-u\widetilde F)|_{\widetilde F}-vZ$ are as follows.  

For $0\le u\le 1$,
$N(u,v)=ve_1$ for $0 \leq v \leq u$; 
for $1\le u\le 2$
\[ N(u,v)= \begin{cases} 0  &\mbox{ for }0 \le v \le u-1,\\
         (v -u + 1)e_1 &\mbox{ for } u-1\le v \le 1,\\
         (v -u + 1)e_1 +(v-1)(\ell_{2(1)}+\ell_{3(1)})&\mbox{ for }1\le v\le u. 
\end{cases}\] 
In addition for $2\le u \le 3$
   \[ N(u,v)=
      \begin{cases}
         0  &\mbox{ for } 0 \le v \le 1,\\
         (v - 3 + u )(\ell_{2(1)} + \ell_{3(1)}) &\mbox{ for } 3-u \le v\le 1,\\
         (v - 3 + u )(\ell_{2(1)} + \ell_{3(1)})+(v-1)e_1 &\mbox{ for } 1\le v\le 4-u.
      \end{cases}\]

We have
{\allowdisplaybreaks\begin{align*}
    S\big(&W_{\bullet,\bullet}^{\widetilde F};Z\big)\leqslant \frac{3}{22}\int_{0}^{3}\int_{0}^{\infty}\mathrm{vol}\Big(P(u)\big\vert_{\widetilde F}-v\ell_{1(2)}\Big)dvdu=\\
    &=\frac{3}{22}\Bigg(\int_0^{1}\int_0^u 2u(u - v)  dvdu + \\
    &+\int_1^{2}\Big(\int_0^{u-1}  -u^2 - v^2 + 6u - 2v - 3 dv  + \int_{u-1}^1 -2uv + 4u - 2  dv + \int_{1}^u  2(2 - v)(u - v) dv\Big) du+\\
    &+\int_2^{3}\Big(\int_0^{3-u}  -v^2 - 4u - 2v + 13 dv + \int_{3-u}^{1} 2u^2 + 4uv + v^2 - 16u - 14v + 31 dv + \\
    &\text{ }\text{ }\text{ }\text{ }\text{ }\text{ }\text{ }\text{ }+\int_{1}^{4-u}  2(u + v - 4)^2 dv\Big)du\Bigg)=\frac{59}{88}<1.
\end{align*}}
This finishes the proof that $\beta(\Xi)>0$ when $Z\neq \widetilde R\big \vert_{\widetilde F}$.

Assume that $Z= \widetilde R\big \vert_{\widetilde F}$, so that $\ord_{Z}\Big(N(u)\big\vert_{\widetilde F}\Big)=1$ when $2\le u\le 3$. 
We have
$$\frac{3}{(-K_X)^3}\int_2^{3}\big(P(u)^{2}\cdot \widetilde F \big)\cdot\ord_{Z}\Big(N(u)\big\vert_{\widetilde F}\Big)du=\frac{9}{22}.$$
As before, denote by $P(u,v)$ and $N(u,v)$ the positive and negative parts of the Zariski decomposition of $\omega^*\pi^*(-K_{X}-u\widetilde F)\big \vert_{\widetilde F}-vZ$. 
When $0\le u\le 1$, 
$N(u,v)=v(e_1 + e_2 + e_3)$ for $0 \le v\le u/2$, 
when $1\le u\le 2$, 
\[ N(u,v)= \begin{cases}
         0 &\mbox{ for }0\le v\le u-1,\\
         (v -u + 1)(e_1 + e_2 + e_3) &\mbox{ for }u-1\le v\le u/2,\\
       \end{cases}\] 
and finally, when $2\le u\le 3$,
\[ N(u,v)= \begin{cases}
         0  &\mbox{ for }0\le v\le 3-u,\\
         (v - 3+u)(\ell_{1(1)}+\ell_{2(1)}+\ell_{3(1)})  &\mbox{ for } 3-u  \le v \le 2-u/2.
      \end{cases}\] 
      We have
{\allowdisplaybreaks\begin{align*}
    S\big(&W_{\bullet,\bullet}^{\widetilde F};Z\big)=\frac{9}{22}+ \frac{3}{22}\int_{0}^{3}\int_{0}^{\infty}\mathrm{vol}\Big(P(u)\big\vert_{\widetilde F}-vZ\Big)dvdu=\\
    &=\frac{9}{22}+\frac{3}{22}\Bigg(\int_0^{1}\int_0^{u/2} 2(u - v)(u - 2v) dvdu + \\
    &+\int_1^{2}\Big(\int_0^{u-1}  -u^2 + v^2 + 6u - 6v - 3  dv + \int_{u-1}^{u/2}  2(u - v)(u - 2v) dv \Big)du+\\
    &+\int_2^{3}\Big(\int_0^{3-u}  2uv + v^2 - 4u - 10v + 13 dv + \int_{3-u}^{2-u/2} (u + 2v - 4)(3u + 2v - 10)  dv \Big)du\Bigg)=\frac{3}{4}<1.
\end{align*}}
      We see that $S_X(\widetilde F)<2$ and $S\big(W^{\widetilde F}_{\bullet,\bullet};Z\big)<1$, so that $\beta(\Xi)>0$. 
\end{proof}
Now we need to consider $G$-invariant prime divisors $\Xi$ whose centre on $Q$ lies on $\Gamma_1$. 
\begin{lemma}\label{lemma:II-EC}
    If $Z= c_{\widetilde X}(\Xi) \subset E_1$, then $\beta(\Xi)>0.$
\end{lemma}
\begin{proof}
Assume that $Z\subset E_1$, then since there is no $G$-fixed point on $Q\subset \DP^4$, $f_1(Z)= c_{Q}(\Xi)$ is the curve $\Gamma_1$. 
Denote by $Q_1\to Q$ the blowup of the line $L_1$ and by $Q_1\to \DP^2$ the morphism induced by the projection $Q\subset \DP^4 \dashrightarrow \DP^2$ away from $L_1$. Let $\widehat X^+\to Q_1$ be the blowup of the proper transform of $\Gamma_1$, then $\widehat X^+\dashrightarrow \widehat X$ is a flop, and there is a morphism $\widetilde X\to \widehat X$. Denote by $\eta$ the composition $\widetilde X\to \widehat X\to Q_1\to \DP^2$.

If $T$ is a~general fiber of $\eta$, $T\cdot \widetilde Z\geqslant 5$, hence, by Lemma~\ref{CorollaryA16}, $\beta(\Xi)>0$.
\end{proof}

\begin{lemma}
\label{lemma:primedivisors-II}  If $\Xi$ is a $G$-invariant prime divisor over $X$ with centre a prime divisor $D_X= c_X(\Xi)$ such that $\beta(\Xi)<0$, then $D_X\in |H_2|$.

\end{lemma}
\begin{proof}
The centre $c_{X}(\Xi)= D_X$ is the $G$-orbit of a minimal log canonical centre of a suitable pair $(X, \frac{3}{4} \mathcal D)$ for $\mathcal D\subset |-K_X|_{\DQ}$ a $G$-invariant linear system, so that $D_X$ is a $G$-invariant irreducible normal surface
 with 
\[ -K_X \sim_{\DQ} \lambda D_X+ \Delta_X \]
for some effective $\DQ$-divisor $\Delta_X$ and rational number $\lambda> \frac{4}{3}$ (see proof of \cite[Theorem 1.52]{Fano21}).
We show that then, $D_X$ is linearly equivalent to $H_2$ (here since $X_1\to X$ is a small map, we also denote $c_{X_1}(\Xi)$ by $D_X$). 

Recall that $\Effb(X_1)= \DR_{\ge 0} [E_1]+ \DR_{\ge 0}[H_2]$, and $H_2\sim 2H_1-E_1$.
If $D_X= E_1$, then 
\[ \Delta \sim 3H_1-(1+\lambda)E_1\sim \frac{3}{2}(2H_1-E_1)+\big(\frac{3}{2}-(1+\lambda)\big)E_1\]
and since $\lambda>\frac{1}{2}$, this is impossible. 

Now assume that $D_X\neq E_1$, so that $f_1(D_X)$ is a $G$-invariant surface on $Q$, and let $d$ be its degree. Since 
  $$3H_1\sim \lambda f_1(D_X)+f_1(\Delta_X),$$
  $3\geq \lambda d$ and $d= 1$ or $d=2$. 
 As there is no $G$-invariant hyperplane section, $d=2$ and 
 \[ \Delta\sim (3-2\lambda)H_1 + (\lambda m_1-1) E_1\]
where $m_1$ is the multiplicity of $f_1(D_X)$ along $\Gamma_1$. 
Since 
\[ \Delta_X \sim \frac{3-2\lambda}{2}(2H_1-E_1)+ \big(\frac{3-2\lambda}{2}+ \lambda m_1-1 \big)E_1\]
we see that $m_1\geq 1$ and $D_X\in |H_2|$. 
\end{proof}

\begin{lemma}
\label{lemma:curves-L-II}
Let $Z= c_{\widetilde X}(\Xi)$ be an irreducible curve that is not contained in $E_1$. Then, $\beta(\Xi)>0$ unless $c_{Q}(\Xi)$ is a line.
\end{lemma}
\begin{proof}
 By Lemma \ref{lemma:primedivisors-II}, a $G$-invariant surface containing $Z$ is either $F$ or the $G$-invariant element of $|H_2|$. We have seen that for such $Z$, $\beta(\Xi)>0$. 
 If $Z\not \subset H_2$, as in the proof of Lemma~\ref{lemma:II-EC}, there is a surjective morphism $\widetilde X\to \DP^2$ and $H_2\cdot Z\leq 2$. Since $L_1$ is in the base locus of $H_2$, this implies that $H_1\cdot Z\leq 1$.
\end{proof}

\begin{maintheorem2*}
$X$ is $K$-polystable.
\end{maintheorem2*}

\begin{proof}
Assume that $X$ is not $K$-polystable, and denote by $\Xi$ a~$G$-invariant prime divisor over $X$ with $\beta(\Xi)\leq 0$. If $c_{X}(\Xi)$ is $0$-dimensional, it is $\{x_0\}$, and $c_{Q}(\Xi)= L_1$, so that $\beta(\Xi)>0$ by Lemma~\ref{II.1c}. If $c_Q(\Xi)$ is a curve and lies on a section $S$ of $|H_2|$, then $\beta(\Xi)>0$ by Lemma~\ref{II.1a}. If $c_{Q}(\Xi)$ is a line, then $\beta(\Xi)>0$ by Lemma~\ref{II.1b} and Lemma~\ref{II.1c}. If $c_{\widetilde X}(\Xi)$ is a curve lying on $E_1$, then $\beta(\Xi)>0$ by Lemma~\ref{lemma:II-EC}, and if $c_{\widetilde X}(\Xi)$ is a curve not lying on $E_1$ and such that $c_{Q}(\Xi)$ is not a line, then $\beta(\Xi)>0$ by Lemma~\ref{lemma:curves-L-II}. This exhausts the cases where $c_{X}(\Xi)$ is $1$-dimensional. 
Assume now that $c_X(\Xi)$ is a prime divisor. Then, by Lemma~\ref{lemma:primedivisors-II}, $\beta(\Xi)>0$ unless $c_X(\Xi)\in |H_2|$. We have seen that $\beta(S)>0$ for $S\in |H_2|$ in the proof of Lemma~\ref{II.1a}, and this concludes the proof. 
\end{proof}
As in the case of Family I, since $\Aut (X)$ is finite, $X$ is K-stable and this implies by openness of K-stability \cite{BL22}: 
\begin{corollary}
  A general one-nodal prime Fano threefold of genus $12$ in Family II is K-stable.   
\end{corollary}

\section{Family III}
Let $X$ be a one-nodal prime Fano threefold of genus $12$ that belongs to Family III of Theorem~\ref{fams} is the midpoint of a Sarkisov link associated to a rational map $V_5\dashrightarrow \DP^1$; we describe the associated birational geometry briefly, see \cite{JRP11, Pro16} and \cite{KP23} for precise statements. 

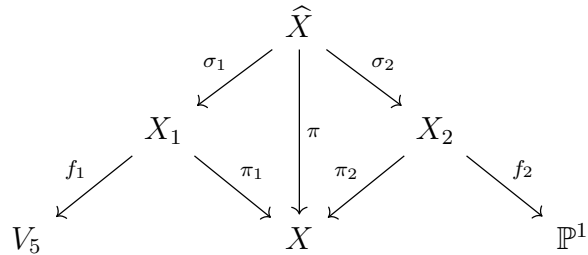
\begin{figure}[h!]
    \centering
    \begin{tikzcd}&& \widehat X \arrow[dl,  "\sigma_1"'] \arrow[dr,  "\sigma_2"] \arrow[dd,  "\pi"]  &&\\
    & X_1  \arrow[dl,  "f_1"'] \arrow[dr,  "\pi_1"]  && X_2 \arrow[dl,  "\pi_2"'] \arrow[dr,  "f_2"]\\
    V_5 && X && \DP^1
    \end{tikzcd}
\end{figure}

Denote by $H_1= \sigma_1^*\big(f_1^*\mathcal{O}_{V_5}(1)\big)$ and $H_2= \sigma_2^*\big(f_2^* \mathcal O_{\DP^1}(1)\big)$, and by $H= \pi^*(-K_X)$ the pullbacks to $\widehat X$ (or to any of the models) of the ample generators of $\Pic (V_5)$, $\Pic (\DP^1)$ and $\Pic (X)$ repsectively. 
The morphism $f_1$ is the blowup of a smooth rational quartic curve $\Gamma_1\subset V_5\subset \DP^6$, and there is a unique bisecant line $L_1$ to $\Gamma_1$. 
The linear system $|H_1-\Gamma_1|$ has dimension $2$,  $\Bs |H_1-\Gamma_1|= \Gamma_1\cup L$, and the rational map associated to $|H_1-\Gamma_1| = |H_2|$ is precisely $V_5\dashrightarrow \DP^1$ induced by the Sarkisov link above. The threefold $X_1$ is weak Fano, and 
\begin{equation*}\label{eq2.1}
    -K_{X_1}\sim H= 2H_1-E_1
\end{equation*}
where $E_1= \Exc f_1$, so that the proper transform of $L_1$ (still denoted $L_1$) is the unique flopping curve on $X_1$. The map $\pi_1$ contracts $L_1$ to a node  $\{x_0\}= \operatorname{Sing}(X)\in X$. Let $\pi\colon \widehat X\to X$ be the blowup of $x_0$, and $\sigma_1$ the induced map to $X_1$; $\chi$ is the Atiyah flop associated to $x_0\in X$ and $L_1= \sigma_1(F)$, where $F= \Exc \pi$.
Then, $\widehat X$ is a weak Fano threefold of $\rho=3$ and we have \cite{KP23}:
\begin{equation*}\label{eq2.2}
 -K_{\widehat X}= H-F \sim H_1+H_2   
\end{equation*}
and from $H\sim 2H_1-E_1 \sim H_1+H_2$, we deduce
\begin{equation*}\label{eq2.3}
    H_2\sim H_1-E_1.
\end{equation*}
The map $f_2$ is a del Pezzo fibration (a Mori fibre space with two-dimensional fibres) of degree $H^2\cdot H_2= 6$. For later reference, the intersection numbers on $\widehat X$ are: 
\begin{align*}
     &H_1^3=5,& &H_1^2\cdot E_1 = 0, & &H_1\cdot E_1^2=-4, & &E_1^3=-6,\\
    &      & &H_1^2\cdot F = 0, & &H_1\cdot F^2=-1, & &F^3=2,\\
    &      & &E_1\cdot F\cdot H_1=0,& &E_1\cdot F^2=-2, & &E_1^2\cdot F=0.
\end{align*}

\subsection{Construction of a member with $\mathbb{G}_m\rtimes \DZ_2$-action} Recall from \cite[Section 5.8]{Fano21} that the quintic threefold $V_5\subset \DP^6$ can be defined scheme theoretically by
$$
\left\{\aligned
&x_4x_5-x_0x_2+x_1^2=0, \\
&x_4x_6-x_1x_3+x_2^2=0,\\
&x_4^2-x_0x_3+x_1x_2=0,\\
&x_1x_4-x_0x_6-x_2x_5=0,\\
&x_2x_4-x_3x_5-x_1x_6=0.\\
\endaligned
\right.
$$
and is endowed with an action of $G= \mathbb{G}_m\rtimes \DZ_2$ defined by the~involution 
$$
\tau\colon \big[x_0:x_1:x_2:x_3:x_4:x_5:x_6\big]\mapsto\big[x_3:x_2:x_1:x_0:x_4:x_6:x_5\big],
$$
and by the~automorphisms $\lambda_s$ 
$$
\big[x_0:x_1:x_2:x_3:x_4:x_5:x_6\big]\mapsto\big[s^3x_0:s^5x_1:s^7x_2:s^9x_3:s^6x_4:s^4x_5:s^8x_6\big].
$$
Consider the curve $\Gamma_1 \subset V_5$ defined by the embedding $\DP^1\hookrightarrow\DP^4$ given by
$$[x:y]\to [0: ix^3y: ixy^3:0:-x^2y^2 :-x^4:-y^4],$$
where $i^2= -1$, then $\Gamma_1$ is a $G$-invariant rational curve of degree $4$. The line $L_1= \{ x_0=x_1=x_2=x_3=x_4=0\}$ is the unique bisecant line to $\Gamma_1$ and it is also $G$-invariant. 
Note that $\Gamma_1$ lies on $\{x_0=x_3=0\}\cap V_5$, and the pencil of hyperplanes containing $\Gamma_1$ is the restriction of \[ \mathcal H=\big\{H_{[\lambda:\mu]}=\{\lambda x_0+\mu x_3=0\}; [\lambda: \mu] \in \DP^1\big\}\] to $V_5$. Denote by $S_{[\lambda: \mu]}= H_{[\lambda: \mu]}\cap V_5$, and note that for any hyperplane in the pencil, $L_1\cup \Gamma_1\subset S_{[\lambda: \mu]}$.
The midpoint $X$ of the Sarkisov link above is endowed with a $G$-action. Finally, denote by $S= \{ x_4=0\} \cap V_5$ the only $G$-invariant hyperplane section of $V_5$, and observe that $S$ has multiplicity $2$ abong $L$, so that $\widetilde S_{[\lambda: \mu]}= H_1-E_1-F$ and $\widetilde S= H_1-2F$ are the proper transforms of $S_{[\lambda: \mu]}$ and $S$ on $\widetilde X$.

\begin{claim}
    The group $\Aut(X)=G$, and in particular, it is reductive.
\end{claim}
\begin{proof}
Since 
\[ G\simeq \Gm \rtimes \DZ_2\subset \Aut(X)\simeq \Aut(V_5; \Gamma_1)\subset \Aut(V_5)= \PGL_2(\DC), \]
by \cite{NvdP}, $\Aut(X)= G$ or $\Aut(X)= \Aut(V_5)= \PGL_2(\DC)$. The second case is impossible because $\Gamma_1$ is not $\Aut(V_5)$-invariant. 
\end{proof}
We will apply Theorem~\ref{equivGstab} to prove that $X$ is K-polystable. To do so, we first describe possible centres of $G$-invariant divisors over $X$. In what follows, $\Xi$ denotes a $G$-invariant prime divisor over $X$. 

\begin{claim}\label{cla3.1}
If the centre of $\Xi$ on $X$ is $0$-dimensional, it is the singular point $c_X(\Xi)= \{x_0\}$.     
\end{claim}
    
\begin{proof}
There is no point of $V_5\subset \DP^6$ fixed by the action of $G$.
\end{proof}

We now consider those $G$-invariant prime divisors over $X$ which have one-dimensional centre $Z=c_{V_5}(\Xi)$ on $V_5$. By \cite[Corollary 5.39]{Fano21}, the $G$-invariant curves on $V_5$ are precisely the line $L_1$, the conic $C$ defined parametrically by $[x:y] \mapsto [x^2:0:0:y^2:xy:0:0]$, the twisted cubic defined parametrically by $[x:y] \mapsto [x^3:x^2y:xy^2:y^3:0:0:0]$
and a family of sextic curves $C_\gamma$ for $\gamma\in \DC^*$ in each of the hyperplane sections $\{x_4=0\}\cap V_5$ and $\{\lambda x_0+ \mu x_3=0\}\cap V_5$.

\begin{lemma}
    Let $\Xi$ be a $G$-invariant prime divisor with centre $Z=c_{V_5}(\Xi)$ a curve. Then $Z= L_1$, $Z=\Gamma_1$ or $\beta(\Xi)>0$.
\end{lemma}
\begin{proof}
Assume to the contrary that $\beta(\Xi)<0$, then by Lemma~\ref{lemma144}, $Z_2= c_{X_2}(\Xi)$ is contained in $\Nklt(X_2, B_{X_2})$ for some $B_{X_2}\sim_{\DQ}-\lambda K_{X_2}$ and rational number $\lambda <\frac{3}{4}$. By Lemma~\ref{CorollaryA15}, the degree $H_2\cdot Z_2\leq 1$, and we exclude the curves with $H_1\cdot Z>1$ by considering $Z_1= c_{\Xi}(Z)$ and its intersections with $\Gamma_1$ and $L_1$. 
If $Z$ is a rational sextic curve constained in $\{x_4=0\}$ or in $\{\lambda x_)+\mu x_3=0\}$, $\Gamma_1\cap L_1= \emptyset$, so $(H_1-E_1)\cdot Z_1= H_2\cdot Z_2$ and $\Gamma_1\cap Z$ consists of at most 2 points, so $H_2\cdot Z_2>1$. Similarly, if $Z=C$ is the $G$-invariant conic or twisted cubic , $C\cap \Gamma_1= C\cap L_1= \emptyset$ and $H_2\cdot Z_2>1$. 
The only possibilities for $Z$ are $L_1$ and $\Gamma_1$. 
\end{proof}

\begin{lemma}
  Let $\Xi$ be a $G$-invariant prime divisor with centre $Z=c_{\widetilde X}(\Xi)$ a curve lying on $\widetilde F$, then $\beta(\Xi)>0$.  
\end{lemma}
\begin{proof}
  Consider the $G$-invariant blowup $\omega:\widetilde{X}\to \widehat X$ of the two flopping lines (these are the transforms of the lines through the singular point on $X$), and denote by $\Lambda= \Lambda_1+ \Lambda_2$ its exceptional divisor $G$. Let $\widetilde F= \omega^* F$, $\widetilde H_1= \omega^* H_1$ and $\widetilde E_1= \omega^* E_1- \Lambda$ be the proper transforms of $F$, $H_1$ and $E_1$. We also have $\widetilde S= \omega^* S- 2 \Lambda$ and $\widetilde S_{[\lambda: \mu]}= \omega^* S_{\lambda: \mu}-\Lambda$. 

If we write the Zariski decomposition of $\omega^*\pi^*(-K_X)-u\widetilde{F}= P(u)+ N(u)$, then $P(u)$ is nef for all $0\leq u \leq 3$ and \[ 
N(u)=
    \begin{cases}
   0&\mbox{ for }0 \le u\le 1,\\
    (u-1)\Lambda&\mbox{ for }2\le u\le 3,\\
    (u-1)\Lambda+(u-2)\widetilde{S}&\mbox{ for }2\le u\le 3.
    \end{cases}
\]
We now compute 
\begin{align} \label{eq:EL-II}
   S_X(\widetilde F)&=\frac{1}{(-K_X)^3}\int_{0}^{\tau}\mathrm{vol}(\omega^*\pi^*(-K_X)-u\widetilde F)du=\\
   &=\frac{1}{22}\Bigg(\int_{0}^{1} 22-2u^3 du + \int_{1}^{2} -6u^2 + 6u + 20 du  + \int_{2}^{3}   2(6 - u)(u - 3)^2 du  \Bigg)=\frac{39}{22}.
\end{align}
So that $\beta(\widetilde F)= A_X(\widetilde F)-  S_X(\widetilde F)=2- \frac{39}{22}= \frac{5}{22}>0$.

We now assume that $Z= c_{\widetilde X}(\Xi)\subset \widetilde F$. The surface $\widetilde F$ is the blowup of $F\simeq \DP^1\times \DP^1$ at two distinct points, that is a del Pezzo surface of degree $6$. Let $\ell_1$ (resp.~$\ell_2$) be the full transform of the ruling of class $(1,0)$ (resp.~$(0,1)$) on $\DP^1\times \DP^1$, and let $e_1, e_2$ be the two exceptional curves. The Mori cone $\NEb(\widetilde F)$ is generated by $e_1$ and $e_2$, and by the $(-1)$-curves $\ell_{i(j)}= \ell_j-e_i$.  

We have:
\begin{multline*}
S\big(W^{\widetilde{F}}_{\bullet,\bullet};Z\big)=\frac{3}{(-K_X)^3}\int_0^{3}\big(P(u)^{2}\cdot \widetilde{F}\big)\cdot\ord_{Z}\Big(N(u)\big\vert_{\widetilde{F}}\Big)du+\\+\frac{3}{(-K_X)^3}\int_0^{3}\int_0^\infty \mathrm{vol}\big(P(u)\big\vert_{\widetilde{F}}-vZ\big)dvdu.
\end{multline*}
 \par As there is no $G$-fixed point on $V_5$, $c_{V_5}(\Xi)$ is not a point and $Z\not \in \{e_1,e_2\}$, so that $Z\not\subset \Lambda\vert_{\widetilde F}$. When in addition, $Z\ne \widetilde{S}\vert_{\widetilde F}$,   $\ord_{Z}\big(N(u)\big\vert_{\widetilde F}\big)=0 $. Write
$$Z\sim \alpha_1e_1 +\alpha_2e_2+ \sum_{\substack{i,j\in \{1,2\}}} \alpha_{ij}\ell_{i(j)},$$ and observe that at least one of the coefficients $\alpha_{ij}\neq 0$. By convexity of volume, if the nonzero coefficient corresponds to the curve $\mathbf{l}$, we get:
$$
S\big(W_{\bullet,\bullet}^{\widetilde F};Z\big)=\frac{3}{22}\int_{0}^{3}\int_{0}^{\infty}\mathrm{vol}\Big(P(u)\big\vert_{\widetilde F}-vZ\Big)dvdu\leqslant \frac{3}{22}\int_{0}^{3}\int_{0}^{\infty}\mathrm{vol}\Big(P(u)\big\vert_{\widetilde F}-v\mathbf{l}\Big)dvdu.
$$
so it is enough to show that the last integral is less than $1$ to deduce a contradiction.
\par {\bf Case 1.} Assume that $Z\neq \widetilde S\vert_{\widetilde F}$, and let $Z\sim \ell_{i(j)}$, for $i,j\in \{1,2\}$. To fix notation, we consider $\ell_{1(2)}$. Denote by $P(u,v)$ and $N(u,v)$ the positive and negative parts of the Zariski decomposition of  $(\omega^*\pi^*(-K_X)-u\widetilde F)|_{\widetilde{F}}-vZ$. \begin{itemize}
    \item For $0\leq u \leq 1$, $N(u,v)= ve_{1}$ for $0\le v\le u$.
    \item For $1\leq u \leq 2$,
    \[ N(u,v)=\begin{cases} 0 &\mbox{ for } 0\le v\le u-1,\\
    (v-u+1)e_{1} &\mbox{ for } u-1\le v\le 1,\\
    (v-u+1)e_{1}+(v-1)\ell_{2(1)} &\mbox{ for } 1\le v\le u.\\
\end{cases}\]
 \item For $2\leq u \leq 3$,
\[ N(u,v)=\begin{cases} 0 &\mbox{ for } 0\le v\le 3-u,\\
    (v-3+u)e_{1} &\mbox{ for } 3-u\le v\le 1,\\
    (v-3+u)e_{1}+(v-1)\ell_{2(1)} &\mbox{ for } 1\le v\le 4- u.\\
\end{cases}\]

\end{itemize}
We have
{\allowdisplaybreaks\begin{align*}
    S\big(&W_{\bullet,\bullet}^{\widetilde{E}_L};Z\big)\leqslant \frac{3}{22}\int_{0}^{3}\int_{0}^{\infty}\mathrm{vol}\Big(P(u)\big\vert_{\widetilde{E}_L}-v\ell_{1(2)}\Big)dvdu=\\
    &=\frac{3}{22}\Bigg(\int_0^{1}\int_0^u 2u(u - v)  dvdu + \\
    &+\int_1^{2}\Big(\int_0^{u-1}   -v^2 + 4u - 2v - 2 dv  + \int_{u-1}^1  u^2 - 2uv + 2u - 1   dv + \int_{1}^u  (u - v + 2)(u - v) dv\Big)du+\\
    &+\int_2^{3}\Big(\int_0^{3-u}   2u^2 + 2uv - v^2 - 16u - 6v + 30 dv + \int_{3-u}^{1} (-3 + u)(3u + 4v - 13) dv + \\
    &+\int_{1}^{4-u}  (u + v - 4)(3u + v - 10) dv\Big)du\Bigg)=\frac{17}{22}<1,
\end{align*}}
which is what we wanted. 
\par {\bf Case 2.} Now assume that  $Z=\widetilde{\mathcal{S}}\vert_{\widetilde F}$ so that $\ord_{Z}\Big(N(u)\big\vert_{\widetilde F}\Big)=1$ on $u\in[2,3]$, and 
\[ \frac{3}{(-K_X)^3}\int_2^{3}\big(P(u)^{2}\cdot \widetilde{F}\big)\cdot\ord_{Z}\Big(N(u)\big\vert_{\widetilde{F}}\Big)du=\frac{4}{11}.\]As before, denote by $P(u,v)$ and $N(u,v)$ the positive and negative part of the Zariski decomposition of $(\omega^*\pi^*(-K_X)-u\widetilde F)\vert_{\widetilde F} -vZ$, so that 
\begin{itemize}
    \item For $0\le u\le 1$, $N(u,v)=v(e_1 + e_2)$ for $0\leq v\leq u/3$.
    \item For $1\le u \le 2$, 
    \[ N(u,v)= \begin{cases}
        0  &\mbox{ for }0 \leq v\leq \frac{u-1}{2},\\
         (v -u + 1)(e_1 + e_2 ) &\mbox{ for } \frac{u-1}{2}\le v\le \frac{u}{3}.\\
    \end{cases}\]
    \item For $2\le u\le 3$, 
    \[N(u,v)= 
      \begin{cases}
         0  &\mbox{ for } 0 \le v\le \frac{3-u}{2},\\
         (v - 3+u)(e_1+e_2)  &\mbox{ for } \frac{3-u}{2}\le v\le 2-\frac{2u}{3}. \end{cases}\]
\end{itemize}We now compute 
 {\allowdisplaybreaks
  \begin{align*}
     S\big(&W_{\bullet,\bullet}^{\widetilde F};Z\big)=\frac{4}{11}+ \frac{3}{22}\int_{0}^{3}\int_{0}^{\infty}\mathrm{vol}\Big(P(u)\big\vert_{\widetilde F} -vZ\Big)dvdu= \\
    &=\frac{4}{11}+\frac{3}{22}\Bigg(\int_0^{1}\int_0^{u/3} 2(u - 2v)(u - 3v) dvdu + \\
    &+\int_1^{2}\Big(\int_0^{(u-1)/2}  -2uv + 4v^2 + 4u - 8v - 2  dv + \int_{(u-1)/2}^{u/3}   2(u - 2v)(u - 3v) dv\Big)du +\\
     &+\int_2^{3}\Big(\int_0^{(3-u)/2}  2u^2 + 6uv + 4v^2 - 16u - 24v + 30 dv + \int_{(3-u)/2}^{2-2u/3} 2(u + 2v - 4)(2u + 3v - 6) dv \Big)du\Bigg)=\\
    &=\frac{25}{44}<1.
 \end{align*}}
 and this finishes the proof since 
 \begin{equation*}
\frac{A_X(\Xi)}{S_X(\Xi)}\geqslant\min\Bigg\{\frac{2}{S_X(\widetilde{F})},\frac{1}{S\big(W^{\widetilde{F}}_{\bullet,\bullet};Z\big)}\Bigg\}>1.
\end{equation*}
\end{proof}

\begin{lemma}
  Let $\Xi$ be a $G$-invariant prime divisor with centre $Z=c_{X_1}(\Xi)$ a curve lying on $E_1$, then $\beta(\Xi)>0$.  
\end{lemma}
\begin{proof}
Assume to the contrary that $\beta(\Xi)<0$, then by Lemma~\ref{lemma144}, $Z_2= c_{X_2}(\Xi)$ is a one-dimensional component of $(X_2, B_{X_2})$, where $B_{X_2}\sim -\lambda K_{X_2}$ for some $\lambda < 3/4$. By Lemma~\ref{CorollaryA13}, $H_2\cdot Z_2\leq 1$. This is impossible as $Z_1= c_{X_1}(\Xi)$ cannot be mapped to a point by $\phi_1$ because there is no $G$-invariant point on $V_5$, and $H_2\cdot Z_1\geq H_1\cdot Z_1\geq 4$. 
\end{proof}
\begin{remark}
    For the sake of completion, observe that $X_1$ itself is divisorially K-polystable. Indeed, for $0\le u\le 1$
    \[ -K_{X_1}-uE_1\sim_{\DQ} H-uE_1 \sim_{\DQ} 2H_1-(1+u)E_1\]
is a mobile divisor, that is the pullback of a nef divisor on $X_2$, and for $u>1$, this divisor is not effective. We have
\begin{align*}
    (-K_{X_1})^3\cdot S_{X_1}(E_1)&= 22\cdot S_{X_1}(E_1)= \int_0^1\operatorname{vol}(2H_1-(1+u)E_1)du=\\&=\int_0^1\big((1-u)(-K_{X_2}+2uH_2\big)^3du= \int_0^1\big((1-u)H+ 2uH_2\big)^3du\\&= \int_0^1(1-u)^2(22+14u)du=\frac{17}{2}
\end{align*}
so that $\beta(E_1)>0$.

\end{remark}

\begin{lemma}
\label{lemma:primedivisors-III}
There is no $G$-invariant irreducible surface $D_X$ such that
$-K_X\sim_{\mathbb{Q}}\lambda D_X+\Delta_X$
for some positive rational number $\lambda>\frac{4}{3}$ and effective $\mathbb{Q}$-divisor $\Delta$.
\end{lemma}
\begin{proof}
Let $D_X$ be such a surface, and denote by $D_1$, $\Delta_1$ the proper transforms of $D_X$ and $\Delta_X$ on $X_1$. We have: 
\[ H\sim_{\DQ} 2H_1+ E_1\sim_{\DQ} \lambda D_1+ \Delta_1.\]
Recall that the pseudo-effective cone $\Effb(X_1)$ is $\DR_{\ge 0}[E_1]+ \DR_{\ge 0}[H_2]$, where $H_2\sim_{\DQ} H_1-E_1$.
If $D_1= E_1$, we see that 
\[ \Delta_1\sim_{\DQ} 2H_2 + (1-\lambda) E_1\]
cannot be an effective divisor. 
We may now assume that $D_1 \in \DR_{\ge 0}[H_1]+ \DR_{\ge 0}[H_2]$, that is $D_1= xH_1-yE_1$, for $x,y\in \DN$ and $x\ge y$. Since $\lambda D_1\le -K_{X_1}$, $\lambda a\leq 2$, so that $a=1$ and $b=0$ or $b=1$. As $D_1$ is mapped to a $G$-invariant surface of $V_5$, $\phi_1(D_1)$ is the hyperplane section $\{x_4=0\}\cap V_5$, and $b=0$.  
Now, $\Delta_1\sim_{\DQ}(2-\lambda)H_1-E_1$, but this cannot be effective as $2-\lambda<1$.
\end{proof}
As in the previous two cases, we conclude: 
\begin{maintheorem3*}
$X$ is $K$-polystable.
\end{maintheorem3*}

This time $X$ is not K-stable as $\Aut(X)= \Gm\rtimes \DZ_2$, but using \cite[Corollary 1.16]{Fano21} (which still holds in the case of a nodal Fano threefold), we conclude: 
\begin{corollary}
  A general one-nodal prime Fano threefold of genus $12$ in Family III is K-polystable.   
\end{corollary}

\bibliography{biblio}
\bibliographystyle{abbrv}

\end{document}